\documentclass[12pt]{amsart}

\usepackage{fullpage}
\usepackage[utf8]{inputenc}
\usepackage[T1]{fontenc}
\usepackage{amsmath,amsfonts,amssymb,amsthm,mathtools}
\usepackage{mathrsfs,dsfont}
\usepackage{stmaryrd}
\usepackage{enumerate}

\usepackage{hyperref}
\usepackage{color}

\newcommand{\E}{\mathbb{E}}
\newcommand{\PP}{\mathbb{P}}
\newcommand{\R}{\mathbb{R}}
\newcommand{\N}{\mathbb{N}}

\newtheorem{theo}{Theorem}[section]

\newtheorem{rem}[theo]{Remark}

\newtheorem{propo}[theo]{Proposition}
\newtheorem{lemma}[theo]{Lemma}

\newtheorem{hyp}[theo]{Assumption}

\begin{document}

\title{Weak convergence rates of splitting schemes for the stochastic Allen-Cahn equation}

\author{Charles-Edouard Br\'ehier}
\address{Univ Lyon, CNRS, Université Claude Bernard Lyon 1, UMR5208, Institut Camille Jordan, F-69622 Villeurbanne, France}
\email{brehier@math.univ-lyon1.fr}

\author{Ludovic Gouden\`ege}
\address{Université Paris-Saclay, CNRS - FR3487, Fédération de Mathématiques de CentraleSupélec, CentraleSupélec, 3 rue Joliot Curie, F-91190 Gif-sur-Yvette, France}
\email{ludovic.goudenege@math.cnrs.fr}


\date{}


\begin{abstract}

This article is devoted to the analysis of the weak rates of convergence of schemes introduced by the authors in a recent work~\cite{Brehier_Goudenege:18}, for the temporal discretization of the stochastic Allen-Cahn equation driven by space-time white noise. The schemes are based on splitting strategies and are explicit. We prove that they have a weak rate of convergence equal to $\frac12$, like in the more standard case of SPDEs with globally Lipschitz continuous nonlinearity.

To deal with the polynomial growth of the nonlinearity, several new estimates and techniques are used. In particular, new regularity results for solutions of related infinite dimensional Kolmogorov equations are established. Our contribution is the first one in the literature concerning weak convergence rates for parabolic semilinear SPDEs with non globally Lipschitz nonlinearities.
\end{abstract}

\maketitle

\section{Introduction}\label{sec:intro}

In this article, we study numerical schemes introduced in by the authors in~\cite{Brehier_Goudenege:18}, for the temporal discretization of the stochastic Allen-Cahn equation,
\[
\frac{\partial X(t,\xi)}{\partial t}=\frac{\partial^2 X(t,\xi)}{\partial \xi^2}+X(t,\xi)-X(t,\xi)^3+\dot{W}(t,\xi),~t\ge 0,\xi\in(0,1),
\]
driven by Gaussian space-time white noise, with homogeneous Dirichlet boundary conditions.

This Stochastic Partial Differential Equation (SPDE) has been introduced in~\cite{Allen_Cahn:79} as a model for a two-phase system driven by the Ginzburg-Landau energy
\[
\mathcal{E}(X) = \int |\nabla X|^{2}+ V(X),
\]
where $X$ is the ratio of the two species densities, and $V(X)=(X^2-1)^2$ is a double well potential. The first term in the energy models the diffusion of the interface between the two pure phases, and the second one pushes the solution to two possible stable states $\pm 1$ (named the pure phases, i.e. minima of $V$). The stochastic version of the Allen-Cahn equation models the effect of thermal perturbations by an additional noise term.

\medskip

The objective of this article is to study weak rates of convergence for two examples of splitting schemes schemes introduced in~\cite{Brehier_Goudenege:18}. Let the SPDE be rewritten in the framework of~\cite{DaPrato_Zabczyk:14}, with $X(t)=X(t,\cdot)$:
\[
dX(t)=AX(t)dt+\bigl(X(t)-X(t)^3\bigr)dt+dW(t),
\]
where $\bigl(W(t)\bigr)_{t\ge 0}$ is a cylindrical Wiener process. If $\Delta t>0$ denotes the time-step size of the integrator, the numerical schemes are defined as
\begin{equation}\label{eq:schemeINTRO}
\begin{aligned}
X_{n+1}^{{\rm exp}}=e^{\Delta tA}\Phi_{\Delta t}(X_n^{{\rm exp}})+\int_{n\Delta t}^{(n+1)\Delta t}e^{((n+1)\Delta t-t)A}dW(t),\\
X_{n+1}^{{\rm imp}}=S_{\Delta t}\Phi_{\Delta t}(X_n^{{\rm imp}})+S_{\Delta t}\bigl(W((n+1)\Delta t)-W(n\Delta t)\bigr).
\end{aligned}
\end{equation}
In both schemes, $\bigl(\Phi_t\bigr)_{t\ge 0}$ is the flow map associated with the ODE $\dot{z}=z-z^3$, which is known exactly, see~\eqref{eq:PhiPsi}, and at each time step, one solves first the equation $dX(t)=\bigl(X(t)-X(t)^3\bigr)dt$, second compute an approximation for the equation $dX(t)=AX(t)dt+dW(t)$. In the first scheme in~\eqref{eq:schemeINTRO}, the second step is also solved exactly, using an exponential integrator and exact sampling of the stochastic integral. In the second scheme, the second step is solved using a linear implicit Euler scheme, with $S_{\Delta t}=(I-\Delta tA)^{-1}$.

The schemes given by~\eqref{eq:schemeINTRO} have been introduced by the authors in~\cite{Brehier_Goudenege:18}, where two preliminary results were established: the existence of moment estimates for the numerical solution, with bounds uniform in the time-step size parameter $\Delta t$, and the mean-square convergence of the scheme, with no order of convergence. In~\cite{Brehier_Cui_Hong:18}, it was established that the first scheme in~\eqref{eq:schemeINTRO} has strong order of convergence $\frac14$, based on a nice decomposition of the error. The contribution of the present article is to prove that both schemes in~\eqref{eq:schemeINTRO} have a weak order of convergence $\frac12$.

Numerical schemes for SPDEs have been extensively studied in the last two decades, see for instance the monographs~\cite{Jentzen_Kloeden:11,Kruse:14,Lord_Powell_Shardlow:14}. We recall that two notions of convergence are usually studied: strong convergence refers to convergence in mean-square sense, whereas weak convergence refers to convergence of distributions. If sufficiently regular test functions are considered, it is usually the case that the weak order of convergence is twice the strong order. For one-dimensional parabolic semilinear SPDEs driven by space-time white noise, the solutions are only H\"older continuous in time with exponent $\alpha<1/4$, hence one expects a strong order of convergence equal to $\frac14$ and a weak order equal to $\frac12$.

In the case of globally Lipschitz continuous nonlinearities, this result has been proved in recent years for a variety of numerical schemes, exploiting different strategies for the analysis of the weak error: analysis of the error using the associated Kolmogorov equation, see~\cite{Andersson_Larsson:16,Brehier_Debussche:17,Debussche:11,Debussche_Printems:09,Wang_Gan:13}, using the mild It\^o formula approach, see~\cite{Conus_Jentzen_Kurniawan:14,Hefter_Jentzen_Kurniawan:16,Jentzen_Kurniawan:15}, or other techniques, see \cite{Andersson_Kruse_Larsson:16,Wang:16}. In the case of SPDEs with non globally Lipschitz continuous nonlinearities, standard integrators, which treat explictly the nonlinearity, cannot be used. Design and analysis of appropriate schemes, is an active field of research, in particular concerning the stochastic Allen-Cahn equation, see \cite{Becker_Gess_Jentzen_Kloeden:17,Brehier_Cui_Hong:18,Brehier_Goudenege:18,Kovacs_Larsson_Lindgren:15_1,Kovacs_Larsson_Lindgren:15_2,Liu_Qiao:17,Liu_Qiao:18,Majee_Prohl:17,Wang:18}. Several strategies may be employed: splitting, taming, split-step methods have been proposed and studied in those references. Only strong convergence results have been obtained so far. Up to our knowledge, only the preliminary results in the PhD Thesis~\cite{Kopec} deal with the analysis of the weak error, for split-step methods using an implicit discretization of the nonlinearity. Hence, the present article is to present detailed analysis of the weak error for a class of SPDEs with non globally Lipschitz continuous nonlinearity.

\medskip

The main results of this article are Theorems~\ref{th:weak_expo} and~\ref{th:weak_implicit}, which may be stated as follows: for sufficiently smooth test functions $\varphi$ (see Assumption~\ref{ass:varphi}), and all $\alpha\in[0,\frac12)$, $T\in(0,\infty)$,
\[
\big|\E\bigl[\varphi(X(N\Delta t))\bigr]-\E\bigl[\varphi(X_N)\bigr]\big|\le C_\alpha(T)\Delta t^\alpha.
\]
In the case of the second scheme in~\eqref{eq:schemeINTRO}, numerical experiments were reported in~\cite{Brehier_Goudenege:18} to motivate and illustrate this convergence result.

As in the case of SPDEs with globally Lipschitz continuous nonlinearities, driven by a cylindrical Wiener process, the order of convergence is $\frac12$, whereas the strong order of convergence is only $\frac14$ in general. However, several points in the analysis are original and need to be emphasized. First, moment estimates for the numerical solution, which are non trivial in the case of numerical schemes for equations with non globally Lipschitz nonlinearities, are required.  They were proved by the authors in~\cite{Brehier_Goudenege:18} for~\eqref{eq:schemeINTRO}. The analysis of the error is based on the decomposition of the error using the solution of related Kolmogorov equations, see~\eqref{eq:Kolmo_Deltat} below and Section~\ref{sec:weak} for a description of the method. Appropriate regularity properties need to be proved for the first and second order spatial derivatives: this is presented in Theorems~\ref{th:Du} and~\ref{th:D2u}, which are auxiliary results in the analysis, but may have a more general interest. The proofs of our main results use different strategies. The proof of Theorem~\ref{th:weak_expo} is shorter, due to the use of a nice auxiliary continuous-time process, and of appropriate temporal and spatial regularity properties of the process, following~\cite{Wang:16} and~\cite{Wang:18}. The proof of Theorem~\ref{th:weak_implicit} follows essentially the same steps as in~\cite{Debussche:11} and related references, in particular a duality formula from Malliavin calculus is used. The estimate of the Malliavin derivative given in Lemma~\ref{lem:malliavin} uses in an essential way the structure as a splitting scheme, and is one of the original results used in this work.

\medskip

This article is organized as follows. Assumptions, equations and numerical schemes are given in Section~\ref{sec:setting}. Our main results, Theorems~\ref{th:weak_expo} and~\ref{th:weak_implicit} are stated in Section~\ref{sec:weak}. The regularity properties for solutions of Kolmogorov equations are stated and proved in Section~\ref{sec:kolmo}. Section~\ref{sec:expo} is devoted to the proof of Theorem~\ref{th:weak_expo}, whereas Section~\ref{sec:impl} is devoted to the proof of Theorem~\ref{th:weak_implicit}.


\section{Setting}\label{sec:setting}

We work in the standard framework of stochastic evolution equations with values in infinite dimensional separable Hilbert and Banach spaces. We refer for instance to~\cite{Cerrai:01, DaPrato_Zabczyk:14} for details. Let $H=L^2(0,1)$, and $E=\mathcal{C}([0,1])$. We use the following notation for inner product and norms: for $x_1,x_2\in H$, $x\in E$,
\[
\langle x_1,x_2\rangle =\int_{0}^{1}x_1(\xi)x_2(\xi)d\xi~,\quad |x_1|_H=\langle x_1,x_1\rangle^{\frac12}~,\quad |x|_E=\underset{\xi\in[0,1]}\max |x(\xi)|.
\]
For $p\in[1,\infty]$, we also use the notation $L^p=L^p(0,1)$ and $|\cdot|_{L^p}$ for the associated norm.

\subsection{Assumptions}

\subsubsection{Linear operator}

Let $A$ denote the unbounded linear operator on $H$, with
\[
\begin{cases}
D(A)=H^2(0,1)\cap H_0^1(0,1),\\
Ax=\partial_\xi^2x,~\forall~x\in D(A).
\end{cases}
\]

Let $e_n=\sqrt{2}\sin(n\pi\cdot)$ and $\lambda_n=n^2\pi^2$, for $n\in\N$. Note that $Ae_n=-\lambda_n e_n$, and that $\bigl(e_n\bigr)_{n\in\N}$ is a complete orthonormal system of $H$. In addition, for all $n\in\N$, $|e_n|_E\le \sqrt{2}$.

The linear operator $A$ generates an analytic semi-group $\bigl(e^{tA}\bigr)_{t\ge 0}$, on $L^p$ for $p\in[2,\infty)$ and on $E$. For $\alpha\in(0,1)$, the linear operators $(-A)^{-\alpha}$ and $(-A)^{\alpha}$ are constructed in a standard way, see for instance~\cite{Pazy}:
\begin{gather*}
(-A)^{-\alpha}=\frac{\sin(\pi \alpha)}{\pi}\int_{0}^{\infty}t^{-\alpha}(tI-A)^{-1}dt~,~(-A)^{\alpha}=\frac{\sin(\pi \alpha)}{\pi}\int_{0}^{\infty}t^{\alpha-1}(-A)(tI-A)^{-1}dt,
\end{gather*}
where $(-A)^{\alpha}$ is defined as an unbounded linear operator on $L^p$. In the case $p=2$, note that
\begin{gather*}
(-A)^{-\alpha}x=\sum_{i\in \N^{\star}}\lambda_i^{-\alpha} \langle x,e_i\rangle e_i, \quad x\in H,\\
(-A)^{\alpha}x=\sum_{i\in \N^{\star}}\lambda_i^\alpha \langle x,e_i\rangle e_i, \quad x\in D_2\bigl((-A)^{\alpha}\bigr)=\left\{x\in H ; \sum_{i=1}^{\infty}\lambda_{i}^{2\alpha}\langle x,e_i\rangle^2<\infty\right\}.
\end{gather*}

We denote by $\mathcal{L}(H)$  the space of bounded linear operators from $H$ to $H$ , with associated norm denoted by $\|\cdot\|_{\mathcal{L}(H)}$. The space of Hilbert-Schmidt operators on $H$ is denoted by $\mathcal{L}_2(H)$, and the associated norm is denoted by $\|\cdot\|_{\mathcal{L}_2(H)}$.

To conclude this section, we state several useful functional inequalities. Inequality~\eqref{eq:inequality1} is a consequence of the Sobolev embedding $W^{2\eta,2}(0,1)\subset \mathcal{C}([0,1])$ when $2\eta>\frac{1}{2}$, and of the equivalence of the norms $W^{2\eta,2}(0,1)$ and $|(-A)^{\eta}\cdot|_{L^2}$. For inequalities~\eqref{eq:inequality2} and~\eqref{eq:inequality3}, we refer to~\cite{Triebel} for the general theory, and to the arguments detailed in~\cite{Brehier_Debussche:17}.

\begin{itemize}
\item For every $\eta>\frac14$, there exists $C_\eta\in(0,\infty)$ such that
\begin{equation}\label{eq:inequality1}
|(-A)^{-\eta}\cdot|_{L^2}\le C_\eta|\cdot|_{L^1}.
\end{equation}

\item For every $\alpha\in(0,\frac12)$, $\epsilon>0$, with $\alpha+\epsilon<\frac12$, there exists $C_{\alpha,\epsilon}\in(0,\infty)$ such that
\begin{equation}\label{eq:inequality2}
|(-A)^{-\alpha-\epsilon}(xy)|_{L^1}\le C_{\alpha,\epsilon}|(-A)^{\alpha+\epsilon}x|_{L^2}|(-A)^{-\alpha}y|_{L^2}.
\end{equation}

\item For every $\alpha\in(0,\frac12)$, $\epsilon>0$, with $\alpha+2\epsilon<\frac12$, there exists $C_{\alpha,\epsilon}\in(0,\infty)$ such that, if $\psi:\R\to\R$ is Lipschitz continuous, 
\begin{equation}\label{eq:inequality3}
|(-A)^{\alpha+\epsilon}\psi(\cdot)|_{L^2}\le C_{\alpha,\epsilon}[\psi]_{\rm Lip}|(-A)^{\alpha+2\epsilon}\cdot|_{L^2},
\end{equation}
where $[\psi]_{\rm Lip}=\underset{z_1\neq z_2}\sup~\frac{|\psi(z_2)-\psi(z_1)|}{|z_2-z_1|}$.

\end{itemize}

These inequalities are used in an essential way to prove Lemma~\ref{lem:Psi} below. In addition, inequality~\eqref{eq:inequality1} is also used in the proof of Theorem~\ref{th:D2u}.

\subsubsection{Wiener process}

Let $\bigl(\Omega,\mathcal{F},\PP\bigr)$ denote a probability space, and consider a family $\bigl(\beta_n\bigr)_{n\in\N}$ of independent standard real-valued Wiener processes. Then set
\[
W(t)=\sum_{n\in\N}\beta_n(t)e_n.
\]
This series does not converge in $H$. However, if $\tilde{H}$ is an Hilbert space, and $L\in\mathcal{L}_2(H,\tilde{H})$ is a linear, Hilbert-Schmidt, operator, then $LW(t)$ is a Wiener process on $\tilde{H}$, centered and with covariance operator $LL^{\star}$.

\subsubsection{Nonlinearities}

For all $t\ge 0$ and all $z\in\R$, define
\begin{equation}\label{eq:PhiPsi}
\Phi_t(z)=\frac{z}{\sqrt{z^2+(1-z^2)e^{-2t}}}~,\quad \Psi_t(z)=\begin{cases} \frac{\Phi_t(z)-z}{t},~t>0,\\ z-z^3,~t=0.\end{cases}
\end{equation}
Note that for all $t\ge 0,z\in \R$, $\Phi_t(z)=z+t\Psi_t(z)$. Moreover, $\bigl(\Phi_t(\cdot)\bigr)_{t\ge 0}$ is the flow map associated with the ODE $\dot{z}=z-z^3=\Psi_0(z)$.

Lemma~\ref{lem:rappel} below states the properties of $\Phi_{\Delta t}$ and $\Psi_{\Delta t}$, and their derivatives, which are used in order to prove well-posedness and moment estimates, and to derive error estimates.

We refer to~\cite{Brehier_Goudenege:18} for a detailed proof (except for the inequality concerning the second order derivative, which is not considered there but is obtained using similar arguments).
\begin{lemma}\label{lem:rappel}
For every $\Delta t_0\in(0,1]$, there exists $C(\Delta t_0)\in(0,\infty)$ such that for all $\Delta t\in[0,\Delta t_0]$, and all $z\in \R$,
\begin{align*}
|\Phi_{\Delta t}'(z)|
&\le e^{\Delta t_0},&\Psi_{\Delta t}'(z)&\le e^{\Delta t_0},\\
|\Psi_{\Delta t}'(z)|&\le C(\Delta t_0)(1+|z|^2),&|\Psi_{\Delta t}''(z)|&\le C(\Delta t_0)(1+|z|),\\
|\Psi_{\Delta t}(z)-\Psi_{0}(z)|
&\le C(\Delta t_0)\Delta t(1+|z|^5).
\end{align*}
In particular, the mapping $\Psi_{\Delta t}$ satisfies the following one-sided Lipschitz condition: for all $z_1,z_2\in\R$,
\[
\bigl(\Psi_{\Delta t}(z_2)-\Psi_{\Delta t}(z_1)\bigr)\bigl(z_2-z_1\bigr)\le e^{\Delta t}|z_2-z_1|^2.
\]
\end{lemma}

Observe also that for $\Delta t>0$, the mapping $\Psi_{\Delta t}$ is of class $\mathcal{C}^\infty$, and admits bounded first and second order derivatives. However, such bounds are not uniform with respect to $\Delta t>0$.

We conclude this section with an auxiliary result, see~\cite{Wang:18} for a similar statement.
\begin{lemma}\label{lem:Psi}
For every $\Delta t_0\in(0,1]$, $\eta\in(\frac{1}{4},1)$, $\alpha\in(0,\frac{1}{2})$, and $\epsilon>0$ such that $\alpha+2\epsilon<\frac{1}{2}$, there exists $C(\Delta t_0,\eta,\alpha,\epsilon)\in(0,\infty)$ such that for all $\Delta t\in[0,\Delta t_0]$ for all $x,y\in H$, with $|(-A)^{\frac{\alpha}{2}+\epsilon}x|_H<\infty$,
\[
|(-A)^{-\eta-\frac{\alpha+\epsilon}{2}}\bigl(\Psi_{\Delta t}'(x)y\bigr)|_H\le C(\Delta t_0,\eta,\alpha,\epsilon)(1+|x|_E)|(-A)^{\frac{\alpha}{2}+\epsilon}x|_H|(-A)^{-\frac{\alpha}{2}}y|_H.
\]
\end{lemma}
\begin{proof}
Using successively the inequalities~\eqref{eq:inequality1} and~\eqref{eq:inequality2},
\begin{align*}
|(-A)^{-\eta-\frac{\alpha+\epsilon}{2}}\bigl(\Psi_{\Delta t}'(x)y\bigr)|_{L^2}&\le C_\eta |(-A)^{-\frac{\alpha+\epsilon}{2}}\bigl(\Psi_{\Delta t}'(x)y\bigr)|_{L^1}\\
&\le C_{\eta,\alpha,\epsilon}|(-A)^{\frac{\alpha+\epsilon}{2}}\Psi_{\Delta t}'(x)|_{L^2}|(-A)^{-\frac{\alpha}{2}}y|_{L^2}.
\end{align*}
The function $\Psi_{\Delta t}'$ is not globally Lipschitz continuous when $\Delta t=0$, and $\underset{\Delta t\to 0}\lim~[\Psi_{\Delta t}']_{\rm Lip}=\infty$. Thus inequality~\eqref{eq:inequality3} cannot be used directly.

However, the derivative $\Psi_{\Delta t}''$ has at most linear growth, uniformly in $\Delta t\in[0,\Delta t_0]$: $|\Psi_{\Delta t}''(z)|\le C(\Delta t_0)(1+|z|)$ for all $z\in\R$. By a straightforward truncation argument, applying inequality~\eqref{eq:inequality3} yields
\[
|(-A)^{\frac{\alpha+\epsilon}{2}}\Psi_{\Delta t}'(x)|_{L^2}\le C_{\alpha,\epsilon}C(\Delta t_0)(1+|x|_E)|(-A)^{\frac{\alpha+2\epsilon}{2}}x|_{L^2}.
\]
This concludes the proof of Lemma~\ref{lem:Psi}.
\end{proof}

\subsection{Stochastic Partial Differential Equations}

The stochastic Allen-Cahn equation with additive space-time white noise perturbation, is
\begin{equation}\label{eq:SPDE_0}
dX(t)=AX(t)dt+\bigl(X(t)-X(t)^3\bigr)dt+dW(t)~,X(0)=x.
\end{equation}
More generally, for $\Delta t\in[0,1]$ introduce the auxiliary equation
\begin{equation}\label{eq:SPDE_Deltat}
dX^{(\Delta t)}=AX^{(\Delta t)}(t)dt+\Psi_{\Delta t}(X^{(\Delta t)}(t))dt+dW(t)~,X^{(\Delta t)}=x.
\end{equation}

If $\Delta t>0$, since $\Psi_{\Delta t}$ is globally Lipschitz continuous, standard fixed point arguments (see for instance~\cite{DaPrato_Zabczyk:14}) imply that, for any initial condition $x\in H$, the SPDE~\eqref{eq:SPDE_Deltat} admits a unique global mild solution $\bigl(X^{(\Delta t)}(t,x)\bigr)_{t\ge 0}$, {\it i.e.} which satisfy
\[
X^{(\Delta t)}(t,x)=e^{tA}x+\int_{0}^{t}e^{(t-s)A}\Psi_{\Delta t}(X^{(\Delta t)}(s,x))ds+\int_{0}^{t}e^{(t-s)A}dW(s),~t\ge 0.
\]

If $\Delta t=0$, proving global well-posedness requires more refined arguments, in particular the use of the one-sided Lipschitz condition (see for instance~\cite{Cerrai:01}). For any initial condition $x\in E$, there exists a unique mild solution $\bigl(X(t,x)\bigr)_{t\ge 0}$ of Equation~\eqref{eq:SPDE_0}, and $X^{(0)}=X$ solves Equation~\eqref{eq:SPDE_Deltat} with $\Delta t=0$:
\[
X^{(0)}(t,x)=e^{tA}x+\int_{0}^{t}e^{(t-s)A}\Psi_{0}(X^{(0)}(s,x))ds+\int_{0}^{t}e^{(t-s)A}dW(s),~t\ge 0.
\]
To simplify notation, we often write $X(t)$ and $X^{(\Delta t)}(t)$ and omit the initial condition $x$.

Let 
\begin{equation}\label{eq:W^A}
W^A(t)=\int_{0}^{t}e^{(t-s)A}dW(s).
\end{equation}
Then (see~\cite[Lemma~$6.1.2$]{Cerrai:01}) for every $T\in(0,\infty)$ and $M\in\N$, there exists $C(T,M)\in(0,\infty)$ such that
\begin{equation}\label{eq:borne_W^A}
\E[\underset{t\in[0,T]}\sup~|W^A(t)|_E^{M}]\le C(T,M).
\end{equation}
Combined the one-sided Lipschitz condition for $\Psi_{\Delta t}$, see Lemma~\ref{lem:rappel},~\eqref{eq:borne_W^A} yields moment estimates for $X^{(\Delta t)}$.
\begin{lemma}\label{lem:moment-Xdeltat}
Let $T\in(0,\infty)$, $\Delta t_0\in(0,1]$ and $M\in\N$. There exists $C(T,\Delta t_0,M)\in(0,\infty)$ such that, for all $\Delta t\in[0,\Delta t_0]$ and $x\in E$,
\[
\E[\underset{t\in[0,T]}\sup~|X^{(\Delta t)}(t,x)|_E^{M}]\le C(T,\Delta t_0,M)(1+|x|_E^M).
\]
\end{lemma}
We refer to~\cite{Brehier_Goudenege:18} for a proof.

\subsection{Kolmogorov equations}

In this section, we introduce functions $u^{(\Delta t)}$ which play a key role in the weak error analysis, and which are solutions of infinite dimensional Kolmogorov equations associated with~\eqref{eq:SPDE_Deltat}.

Let $\varphi:H\to \R$ be a function of class $\mathcal{C}^2$, with bounded first and second order derivatives. Let $\Delta t\in(0,1]$. Note that we do not consider the case $\Delta t=0$ in this section, the reason why will be clear below.

For every $t\ge 0$, set
\begin{equation}\label{eq:u}
u^{(\Delta t)}(t,x)=\E\bigl[\varphi(X^{(\Delta t)}(t,x))\bigr].
\end{equation}
Formally, $u^{(\Delta t)}$ is solution of the Kolmogorov equation associated with~\eqref{eq:SPDE_Deltat}:
\begin{equation}\label{eq:Kolmo_Deltat}
\frac{\partial u^{(\Delta t)}(t,x)}{\partial t}=\mathcal{L}^{(\Delta t)}u^{(\Delta t)}(t,x)=\langle Ax+\Psi_{\Delta t}(x),Du^{(\Delta t)}(t,x)\rangle+\frac{1}{2}\sum_{j\in\N}D^2u^{(\Delta t)}(t,x).(e_j,e_j),
\end{equation}
where the first order spatial derivative is identified as an element of $H$ thanks to Riesz Theorem.

A rigorous meaning can be given using an appropriate regularization procedure. Since $\Psi_{\Delta t}$ is globally Lipschitz continuous for fixed $\Delta t>0$, this may be performed by a standard spectral Galerkin approximation. However, this choice would not allow us to pass to the limit $\Delta t\to 0$ in estimates and keep bounds uniform in the regularization parameter. Instead, we propose to replace the noise $dW(t)$ in~\eqref{eq:SPDE_Deltat} by $e^{\delta A}dW(t)$, with the regularization parameter $\delta>0$. Rigorous computations are performed with fixed $\delta>0$. Section~\ref{sec:kolmo} deals with regularity properties for $Du^{(\Delta t)}(t,x)$ and $D^2u^{(\Delta t)}(t,x)$, which allow us to obtain bounds which are uniform with respect to $\delta$, hence passing to the limit $\delta\to 0$ is allowed. 

To simplify notation, we do not mention the regularization parameter $\delta$ in the computations and statements below.

\subsection{Splitting schemes}

We are now in position to rigorously define the numerical schemes which are studied in this article.

A first scheme is defined as:
\begin{equation}\label{eq:scheme_expo}
X_{n+1}=e^{\Delta t A}\Phi_{\Delta t}(X_n)+\int_{n\Delta t}^{(n+1)\Delta t}e^{(n\Delta t-t)A}dW(t).
\end{equation}

A second scheme is defined as:
\begin{equation}\label{eq:scheme_implicit}
X_{n+1}=S_{\Delta t}\Phi_{\Delta t}(X_n)+S_{\Delta t}\bigl(W((n+1)\Delta t)-W(n\Delta t)\bigr),
\end{equation}
where $S_{\Delta t}=(I-\Delta tA)^{-1}$.

The schemes are constructed following a Lie-Trotter splitting strategy for the SPDE~\eqref{eq:SPDE_0}. Firstly, the equation $dX(t)=\Psi_0(X(t))dt$ is solved explictly using the flow map at time $t=\Delta t$, namely $\Phi_{\Delta t}$. Secondly, the equation $dX(t)=AX(t)dt+dW(t)$ is either solved exactly in scheme~\eqref{eq:scheme_expo}, or using a linear implicit Euler scheme in~\eqref{eq:scheme_implicit}.

As already emphasized in~\cite{Brehier_Cui_Hong:18,Brehier_Goudenege:18}, observe that~\eqref{eq:scheme_expo} and~\eqref{eq:scheme_implicit} can be interpreted as integrators for the auxiliary equation~\eqref{eq:SPDE_Deltat} with nonlinear coefficient $\Psi_{\Delta t}$: respectively
\begin{gather*}
X_{n+1}=e^{\Delta t A}X_n+\Delta te^{\Delta tA}\Psi_{\Delta t}(X_n)+\int_{n\Delta t}^{(n+1)\Delta t}e^{(n\Delta t-t)A}dW(t),\\
X_{n+1}=S_{\Delta t}X_n+\Delta tS_{\Delta t}\Psi_{\Delta t}(X_n)+S_{\Delta t}\bigl(W((n+1)\Delta t)-W(n\Delta t)\bigr),
\end{gather*}
where in~\eqref{eq:scheme_expo}, an exponential Euler integrator is used, whereas in~\eqref{eq:scheme_implicit} a semi-implicit integrator is used.

Moment estimates are available. We refer to~\cite{Brehier_Goudenege:18} for a proof.
\begin{lemma}\label{lem:moment-Xn}
Let $T\in(0,\infty)$, $\Delta t_0\in(0,1]$ and $M\in\N$. There exists $C(T,\Delta t_0,M)\in(0,\infty)$ such that, for all $\Delta t\in[0,\Delta t_0]$ and $x\in E$,
\[
\E\bigl[\underset{n\in\N,n\Delta t\le T}\sup~|X_n|_E^{M}\bigr]\le C(T,\Delta t_0,M)(1+|x|_E^M).
\]
\end{lemma}


\section{Weak convergence results}\label{sec:weak}

This section is devoted to the statement of the main result of this article: the numerical schemes~\eqref{eq:scheme_expo} and~\eqref{eq:scheme_implicit} have a weak convergence rate equal to $\frac{1}{2}$, see Theorems~\ref{th:weak_expo} and~\ref{th:weak_implicit} below respectively.

The main difficulty and novelty of this contribution is the treatment of SPDEs with non globally Lipschitz continuous nonlinear coefficient. Up to our knowledge, except in the PhD Thesis~\cite{Kopec} (where split-step schemes based on an implicit discretization of the nonlinear term, for more general polynomial coefficients, are considered), there is no analysis of weak rates of convergence for that situation in the literature.

Strong convergence of numerical schemes~\eqref{eq:scheme_expo} and~\eqref{eq:scheme_implicit} is proved in~\cite{Brehier_Goudenege:18}, without rate. In~\cite{Brehier_Cui_Hong:18}, the strong rate of convergence $\frac{1}{4}$ is proved for the scheme~\eqref{eq:scheme_expo}.

Test functions satisfy the following condition.
\begin{hyp}\label{ass:varphi}
The function $\varphi:H\to \R$ is of class $\mathcal{C}^2$, and has bounded first and second order derivatives:
\begin{align*}
\|\varphi\|_{1,\infty}&=\underset{x\in H,h\in H,|h|_H=1}\sup~|D\varphi(x).h|<\infty,\\
\|\varphi\|_{2,\infty}&=\|\varphi\|_{1,\infty}+\underset{x\in H,h,k\in H,|h|_H=|k|_H=1}\sup~|D^2\varphi(x).(h,k)|<\infty.
\end{align*}
\end{hyp}

We are now in position to state our main results.
\begin{theo}\label{th:weak_expo}
Let $(X_{n})_{n\in\N}$ be defined by the scheme \eqref{eq:scheme_expo}.

Let $T\in(0,\infty)$, $\Delta t_0\in(0,1]$ and $x\in E$. For all $\alpha\in[0,\frac12)$, there exists $C_\alpha(T,\Delta t_0,|x|_E)\in(0,\infty)$ such that the following holds true.

Let $\varphi$ satisfy Assumption~\ref{ass:varphi}. For all $\Delta t\in(0,\Delta t_0]$ and $N\in\N$, such that $T=N\Delta t$,
\begin{equation}\label{eq:th_weak_expo}
\big|\E\bigl[\varphi(X(N\Delta t))\bigr]-\E\bigl[\varphi(X_{N})\bigr]\big|\le C_{\alpha}(T,\Delta t_0,|x|_E)\|\varphi\|_{1,\infty}\Delta t^\alpha,
\end{equation}
\end{theo}

\begin{theo}\label{th:weak_implicit}
Let $(X_{n})_{n\in\N}$ be defined by the scheme \eqref{eq:scheme_implicit}.

Let $T\in(0,\infty)$, $\Delta t_0\in(0,1]$ and $x\in E$. For all $\alpha\in[0,\frac12)$, there exists $C_\alpha(T,\Delta t_0,|x|_E)\in(0,\infty)$ such that the following holds true.

Let $\varphi$ satisfy Assumption~\ref{ass:varphi}. For all $\Delta t\in(0,\Delta t_0]$ and $N\in\N$, such that $T=N\Delta t$,
\begin{equation}\label{eq:th_weak_imp}
\big|\E\bigl[\varphi(X(N\Delta t))\bigr]-\E\bigl[\varphi(X_{N})\bigr]\big|\le C_{\alpha}(T,\Delta t_0,|x|_E)\|\varphi\|_{2,\infty}\Delta t^\alpha.
\end{equation}
\end{theo}

Theorems~\ref{th:weak_expo} and~\ref{th:weak_implicit} are natural generalizations of the results obtained for SPDEs with globally Lipschitz continuous nonlinear coefficient, see for instance~\cite{Debussche:11}. We obtain the same weak order of convergence $\frac{1}{2}$, which is twice the strong order of convergence.

\begin{rem}
The regularity of the function $\varphi$ is essential to get a weak order of convergence which is twice the strong order, as proved in~\cite{Brehier:17}. If one wants to replace $\|\varphi\|_{2,\infty}$ by $\|\varphi\|_{1,\infty}$ in the right-hand side of~\eqref{eq:th_weak_imp}, the order of convergence has to be replaced by $\frac{\alpha}{2}$, even in the absence of nonlinear coefficient.

In the right-hand side of~\eqref{eq:th_weak_expo}, it is sufficient to control only $\|\varphi\|_{1,\infty}$. This is due to an appropriate decomposition of the weak error. This is not in contradiction with~\cite{Brehier:17}: in the absence of nonlinear coefficient, the weak error is equal to zero.
\end{rem}

Proving Theorem~\ref{th:weak_expo} and~\ref{th:weak_implicit} is the aim of the remainder of the article. The strategy consists in decomposing the weak error as follows:
\begin{align*}
\big|\E\bigl[\varphi(X(N\Delta t))\bigr]-\E\bigl[\varphi(X_{N})\bigr]\big|&\le \big|\E\bigl[\varphi(X(N\Delta t))\bigr]-\E\bigl[\varphi(X^{(\Delta t)}(N\Delta t))\bigr]\big|\\
&+\big|\E\bigl[\varphi(X^{(\Delta t)}(N\Delta t))\bigr]-\E\bigl[\varphi(X_{N})\bigr]\big|.
\end{align*}

The first error term is estimated using the following result, quoted from~\cite{Brehier_Goudenege:18}, combined with globally Lipschitz continuity of $\varphi$ induced by Assumption~\ref{ass:varphi}.
\begin{propo}\label{propo:error_continu}
Let $T\in(0,\infty)$, $\Delta t_0\in(0,1)$ and $x\in E$. There exists $C(T,\Delta t_0,|x|_E)\in(0,\infty)$ such that for all $\Delta t\in (0,\Delta t_0]$,
\[
\underset{t\in[0,T]}\sup~\E\bigl[\big|X(t)-X^{(\Delta t)}(t))\big|_H\bigr]\le C(T,\Delta t_0,|x|_E)\Delta t.
\]
\end{propo}

The treatment of the second error term requires more subtle arguments. First, thanks to~\eqref{eq:u}, and a telescoping sum argument,
\begin{align*}
\E\bigl[\varphi(X^{(\Delta t)}(n\Delta t))\bigr]-\E\bigl[\varphi(X_{n})\bigr]&=u^{(\Delta t)}(n\Delta t,x)-\E\bigl[u^{(\Delta t)}(0,X_n)\bigr]\\
&=\sum_{k=0}^{n-1}\Bigl(\E\bigl[u^{(\Delta t)}((n-k)\Delta t,X_k)-u^{(\Delta t)}((n-k-1)\Delta t,X_{k+1})\bigr]\Bigr).
\end{align*}

The details then depend on the numerical scheme. First, an auxiliary continuous-time process $\tilde{X}$ is introduced, see~\eqref{eq:scheme_expo_tilde} and~\eqref{eq:scheme_implicit_tilde}. It satisfies $\tilde{X}(k\Delta t)=X_k$ for all $k\in\N$. Second, It\^o formula is applied, and the Kolmogorov equation~\eqref{eq:Kolmo_Deltat} is used. Theorem~\ref{th:weak_expo} (numerical scheme given by~\eqref{eq:scheme_expo}) is proved in Section~\ref{sec:expo}. Theorem~\ref{th:weak_implicit} (numerical scheme given by~\eqref{eq:scheme_implicit}) is proved in Section~\ref{sec:impl}.

Spatial derivatives $Du^{(\Delta t)}(t,x)$ and $D^2u^{(\Delta t)}(t,x)$ appear in the expansion of the error obtained following this standard strategy. In infinite dimension, see~\cite{Andersson_Larsson:16,Brehier_Debussche:17,Debussche:11,Wang_Gan:13}, appropriate regularity properties are required to obtain the weak order of convergence $\frac12$. They are studied in Section~\ref{sec:kolmo}.


\section{Regularity properties for solutions of Kolmogorov equations}\label{sec:kolmo}

This section is devoted to state and prove regularity properties of the function $u^{(\Delta t)}$, defined by~\eqref{eq:u}, solution of the Kolmogorov equation \eqref{eq:Kolmo_Deltat} associated to the auxiliary equation \eqref{eq:SPDE_Deltat}. The main difficulty and novelty is due to the poor regularity property of $\Psi_{\Delta t}$: even if for fixed $\Delta t$, $\Psi_{\Delta t}$ is globally Lipschitz continuous, there is no bound which is uniform in $\Delta t>0$, since $\Psi_0$ is polynomial of degree $3$.

Theorems~\ref{th:Du} and~\ref{th:D2u} below are the main results of this section, and they are of interest beyond the analysis of weak convergence rates. They are natural generalizations in a non-globally Lipschitz framework of the estimates provided in~\cite{Debussche:11}, and extended in~\cite{Brehier_Debussche:17} with nonlinear diffusion coefficients. We emphasize that the right-hand sides in the estimates~\eqref{eq:th_Du} and~\eqref{eq:th_D2u} do not depend on $\Delta t$.

\begin{theo}\label{th:Du}
Let $T\in(0,\infty)$ and $\Delta t_0\in(0,1]$.  For every $\alpha\in[0,1)$, there exists $C_\alpha(T,\Delta t_0)\in(0,\infty)$ such that, for all $\Delta t\in(0,\Delta t_0)$, $x\in E$, $h\in H$ and $t\in(0,T]$,
\begin{equation}\label{eq:th_Du}
|Du^{(\Delta t)}(t,x).h|\le \frac{C_\alpha(T,\Delta t_0)(1+|x|_E^2)\|\varphi\|_{1,\infty}}{t^\alpha}|(-A)^{-\alpha}h|_H.
\end{equation}
\end{theo}

This may be interpreted as a regularization property: the assumption $Du^{(\Delta t)}(0,x)\in H$ implies that, for positive $t$, $(-A)^{\alpha}Du^{(\Delta t)}(t,x)\in H$ with $\alpha\in(0,1)$.

\begin{theo}\label{th:D2u}
Let $T\in(0,\infty)$ and $\Delta t_0\in(0,1]$. For every $\beta,\gamma\in[0,1)$, with the condition $\beta+\gamma<1$, there exists $C_{\beta,\gamma}(T,\Delta t_0)\in(0,\infty)$ such that, for all $\Delta t\in(0,\Delta t_0)$, $x\in E$, $h,k\in H$ and $t\in(0,T]$,
\begin{equation}\label{eq:th_D2u}
|D^2u^{(\Delta t)}(t,x).(h,k)|\le \frac{C_{\beta,\gamma}(T,\Delta t_0)(1+|x|_E^7)\|\varphi\|_{2,\infty}}{t^{\beta+\gamma}}|(-A)^{-\beta}h|_H|(-A)^{-\gamma}k|_H.
\end{equation}
\end{theo}

\begin{rem}
Results similar to Theorems~\ref{th:Du} and~\ref{th:D2u} are studied in~\cite{Kopec}, with different techniques.

Note that we obtain stronger results. In Theorem~\ref{th:Du}, one may choose $\alpha\in[0,1)$ instead of $\alpha\in[0,\frac12)$. In Theorem~\ref{th:D2u}, one may choose $\beta,\gamma\in[0,1)$ such that $\beta+\gamma<1$, instead of $\beta,\gamma\in[0,\frac12)$. The stronger statements are useful below, to simplify the treatments of several error terms.
\end{rem}

Expressions of $Du^{(\Delta t)}(t,x)$ and $D^2u^{(\Delta t)}(t,x)$ are given below: for $h,k\in H$, $x\in H$, and $t\ge 0$,
\begin{equation}\label{eq:Du-D2u}
\begin{aligned}
Du^{(\Delta t)}(t,x).h&=\E\bigl[D\varphi(X^{(\Delta t)}(t,x)).\eta^h(t,x)\bigr],\\
D^2u^{(\Delta t)}(t,x).(h,k)&=\E\bigl[D\varphi(X^{(\Delta t)}(t,x)).\zeta^{h,k}(t,x)\bigr]\\
&+\E\bigl[D^2\varphi(X^{(\Delta t)}(t,x)).(\eta^h(t,x),\eta^k(t,x))\bigr],
\end{aligned}
\end{equation}
where the processes $\eta^h(\cdot,x)$ and $\zeta^{h,k}(\cdot,x)$ are the solutions the random PDEs
\begin{equation}\label{eq:eta}
\frac{d\eta^h(t,x)}{dt}=A\eta^h(t,x)+\Psi_{\Delta t}'(X^{(\Delta t)}(t,x))\eta^h(t,x),
\end{equation}
with initial condition $\eta^h(0,x)=h$, and
\begin{equation}\label{eq:zeta}
\frac{d\zeta^{h,k}(t,x)}{dt}=A\zeta^{h,k}(t,x)+\Psi_{\Delta t}'(X^{(\Delta t)}(t,x))\zeta^{h,k}(t,x)+\Psi_{\Delta t}''(X^{(\Delta t)}(t,x))\eta^h(t,x)\eta^k(t,x),
\end{equation}
with initial condition $\zeta^{h,k}(0,x)=0$.

To simplify notation, the parameter $\Delta t$ is omitted in the notation for $\eta^h(t,x)$ and $\zeta^{h,k}(t,x)$.

\subsection{Proof of Theorem~\ref{th:Du}}

Following the strategy in~\cite{Brehier_Debussche:17}, introduce the auxiliary process
\[
\tilde{\eta}^h(t,x)=\eta^{h}(t,x)-e^{tA}h.
\]
Then, thanks to~\eqref{eq:Du-D2u},
\[
Du^{(\Delta t)}(t,x).h=\E\bigl[D\varphi(X^{(\Delta t)}(t,x)).(e^{tA}h)\bigr]+\E\bigl[D\varphi(X^{(\Delta t)}(t,x)).\tilde{\eta}^h(t,x)\bigr],
\]
and thanks to~\eqref{eq:eta},
\[
\frac{d\tilde{\eta}^h(t,x)}{dt}=A\tilde{\eta}^h(t,x)+\Psi_{\Delta t}'(X^{(\Delta t)}(t,x))\tilde{\eta}^h(t,x)+\Psi_{\Delta t}'(X^{(\Delta t)}(t,x))e^{tA}h~,\quad \tilde{\eta}^h(0)=0.
\]

On the one hand, for any $\alpha\in[0,1)$, $h\in H$ and $t>0$,
\begin{equation}\label{eq:Du_hand1}
\big|\E\bigl[D\varphi(X^{(\Delta t)}(t,x)).(e^{tA}h)\bigr]\big|\le \|\varphi\|_{1,\infty}|e^{tA}h|_H\le \frac{C_\alpha}{t^{\alpha}}\|\varphi\|_{1,\infty}|(-A)^{-\alpha}h|_H.
\end{equation}
On the other hand, the process $\tilde{\eta}^h$ may be expressed as
\begin{equation}\label{eq:tilde_eta}
\tilde{\eta}^h(t,x)=\int_{0}^{t}U(t,s)\bigl(\Psi_{\Delta t}'(X^{(\Delta t)}(s,x))e^{sA}h\bigr)ds,
\end{equation} 
where $\bigl(U(t,s)h\bigr)_{t\ge s}$ solves, for every $h\in H$,
\begin{equation}\label{eq:U}
\frac{d U(t,s)h}{dt}=\bigl(A+\Psi_{\Delta t}'(X^{(\Delta t)}(t,x))\bigr)U(t,s)h~,\quad U(s,s)h=h.
\end{equation}
A straightforward energy estimate, using the one-sided Lipschitz condition for $\Psi_{\Delta t}'$, yields
\[
\frac{1}{2}\frac{d|U(t,s)h|^2}{dt}\le (e^{\Delta t_0}-\lambda_1)|U(t,s)h|^2,
\]
thus by Gronwall Lemma $|U(t,s)h|^2\le C(T,\Delta t_0)|h|^2$, for all $s\le t\le T$.

Thus, thanks to~\eqref{eq:tilde_eta}, for $\alpha\in[0,1)$,
\begin{align*}
|\tilde{\eta}^h(t,x)|_H&\le C(T,\Delta t_0)\int_{0}^{t}|\Psi_{\Delta t}'(X^{(\Delta t)}(s,x))e^{sA}h|_Hds\\
&\le C(T,\Delta t_0)\underset{s\in[0,T]}\sup~|\Psi_{\Delta t}'(X^{(\Delta t)}(s,x))|_E \int_{0}^{t}\frac{C_\alpha}{s^{\alpha}}ds |(-A)^{-\alpha}h|_H.
\end{align*}
Thanks to Lemmas~\ref{lem:rappel} and~\ref{lem:moment-Xdeltat},
\begin{equation}\label{eq:Du_hand2}
\big|\E\bigl[D\varphi(X^{(\Delta t)}(t,x)).\tilde{\eta}^h(t,x)\bigr]\big|\le \frac{C_\alpha(T,\Delta t_0)}{t^{\alpha}}(1+|x|_E^2)\|\varphi\|_{1,\infty}|(-A)^{-\alpha}h|_H.
\end{equation}
Combining~\eqref{eq:Du_hand1} and~\eqref{eq:Du_hand2} concludes the proof of~\eqref{eq:th_Du}.

For future reference, note that for $t\in(0,T]$,
\begin{equation}\label{eq:estimee_eta}
|\eta^h(t,x)|_H\le \frac{C_\alpha(T,\Delta t_0)}{t^\alpha}|(-A)^{-\alpha}h|_H (1+\underset{s\in[0,T]}\sup~|\Psi_{\Delta t}'(X^{(\Delta t)}(s,x))|_E),
\end{equation}

\subsection{Proof of Theorem~\ref{th:D2u}}

Thanks to~\eqref{eq:Du-D2u},
\[
|D^2u^{(\Delta t)}(t,x).(h,k)|\le \|\varphi\|_{1,\infty}\E[|\zeta^{h,k}(t,x)|_H\bigr]+\|\varphi\|_{2,\infty}\E[|\eta^h(t,x)|_H|\eta^k(t,x)|_H].
\]

On the one hand, thanks to~\eqref{eq:estimee_eta}, using Lemma \ref{lem:rappel} and Lemma \ref{lem:moment-Xdeltat}, we obtain
\begin{equation}\label{eq:D2u_hand1}
\E[|\eta^h(t,x)|_H|\eta^k(t,x)|_H]\le \frac{C_{\beta,\gamma}(T,\Delta t_0)(1+|x|_E^4)}{t^{\beta+\gamma}}|(-A)^{-\beta}h|_H|(-A)^{-\gamma}|_H,
\end{equation}
for all $\beta,\gamma\in[0,1)$.

On the other hand, using~\eqref{eq:U}, from the equation~\eqref{eq:zeta}, the process $\zeta^{h,k}$ may be expressed as
\begin{equation}\label{eq:zeta_mild}
\zeta^{h,k}(t,x)=\int_{0}^{t}U(t,s)\bigl(\Psi_{\Delta t}''(X^{(\Delta t)}(s,x))\eta^h(s,x)\eta^k(s,x)\bigr)ds.
\end{equation}

Thanks to~\eqref{eq:inequality1},
\[
\big|(-A)^{-\frac12}\bigl(\Psi_{\Delta t}''(X^{(\Delta t)}(s,x))\eta^h(s,x)\eta^k(s,x)\bigr)\big|_H
\le C\big|\Psi_{\Delta t}''(X^{(\Delta t)}(s,x))\big|_E |\eta^h(s,x)|_H |\eta^k(s,x)|_H.
\]
The following result allows us to use this inequality in~\eqref{eq:zeta_mild}.
\begin{lemma}\label{lem:aux-eta}
Let $T\in(0,\infty)$ and $\Delta t_0\in(0,1)$. There exists $C(T,\Delta t_0)\in(0,\infty)$ such that, for all $\Delta t\in(0,\Delta t_0)$, $x\in E$, $h\in H$, $0\le s<t\le T$,
\[
|U(t,s)h|_H\le \frac{C(T,\Delta t_0)}{(t-s)^{\frac12}}\bigl(1+\underset{0\le r\le T}\sup~|\Psi_{\Delta t}'(X^{(\Delta t)}(r,x))|_{E}\bigr)|(-A)^{-1/2}h|_H.
\]
\end{lemma}

The proof of Lemma~\ref{lem:aux-eta} is postponed to Section~\ref{sec:proof_lem_aux-eta}. We refer to~\cite{DaPrato_Debussche:17} for a similar result and the idea of the proof.

Thanks to Lemma~\ref{lem:aux-eta} and to~\eqref{eq:zeta_mild}, we get
\begin{align*}
\E|\zeta^{h,k}(t,x)|_{H} \leq \int_{0}^{t}
\frac{C(T,\Delta t_0)}{(t-s)^{\frac12}}&\bigl(1+\underset{0\le r\le T}\sup~|\Psi_{\Delta t}'(X^{(\Delta t)}(r,x))|_{E}\bigr)\\
&\big|\Psi_{\Delta t}''(X^{(\Delta t)}(s,x))\big|_E |\eta^h(s,x)|_H |\eta^k(s,x)|_Hds.
\end{align*}
Thanks to~\eqref{eq:estimee_eta}, and Lemmas~\ref{lem:rappel} and~\ref{lem:moment-Xdeltat}, we obtain
\begin{align*}
\E|\zeta^{h,k}(t,x)|_{H} \leq C_{\beta,\gamma}(T,\Delta t_{0})(1+|x|^{7}_{E})\int_{0}^{t}
\frac{1}{(t-s)^{\frac12}s^{\beta+\gamma}}ds \ |(-A)^{\beta}h|_{H}|(-A)^{\gamma}k|_{H},
\end{align*}
hence
\begin{equation}\label{eq:D2u_hand2}
\E|\zeta^{h,k}(t,x)|_H\le \frac{C_{\beta,\gamma}(T,\Delta t_0)(1+|x|_E^7)}{t^{\beta+\gamma}}|(-A)^{-\beta}h|_H|(-A)^{-\gamma}k|_H.
\end{equation}
Combining~\eqref{eq:D2u_hand1} and~\eqref{eq:D2u_hand2}, thanks to~\eqref{eq:Du-D2u}, then concludes the proof of~\eqref{eq:th_D2u}.

\subsection{Proof of Lemma~\ref{lem:aux-eta}}\label{sec:proof_lem_aux-eta}

We claim that for all $0\le s\le T$ and $h\in H$,
\begin{equation}\label{eq:lemma_integral}
\int_{s}^{T}|U(r,s)h|_H^2 dr \le C(T,\Delta t_0)\bigl(1+\underset{0\le r\le T}\sup~|\Psi_{\Delta t}'(X^{(\Delta t)}(r,x))|_{E}\bigr)|(-A)^{-1/2}h|_H^2.
\end{equation}

Lemma~\eqref{lem:aux-eta} is then a straightforward consequence of~\eqref{eq:lemma_integral}, using the mild formulation
\[
U(t,s)h=e^{(t-s)A}h+\int_{s}^{t}e^{(t-r)A}\Psi_{\Delta t}'(X^{(\Delta t)}(s,x))U(r,s)h ds.
\]

It remains to prove~\eqref{eq:lemma_integral}. Let $s\in[0,T]$ be fixed, and define
\[
\mathcal{U}_s:h\in H\mapsto \bigl(U(t,s)h\bigr)_{s\le t\le T}\in L^2(s,T;H).
\]
Introduce $\mathcal{U}_s^\star:L^2(s,T;H)\to H$ the adjoint of $\mathcal{U}_s$. Then, by a duality argument, the claim~\eqref{eq:lemma_integral} is a straightforward consequence of the following estimate: for all $F\in L^2(s,T;H)$,
\begin{equation}\label{eq:duality}
|(-A)^{\frac12}\mathcal{U}_s^\star F|^2\le C(T,\Delta t_0)\bigl(1+\underset{0\le r\le T}\sup~|\Psi_{\Delta t}'(X^{(\Delta t)}(r,x))|_{E}\bigr)\int_{s}^{T}|F(r)|_H^2dr.
\end{equation}
To prove~\eqref{eq:duality}, let $F\in L^2(s,T;H)$, and observe that $\mathcal{U}_s^\star F=\xi_s(s)$, where $\bigl(\xi_s(t)\bigr)_{s\le t\le T}$ is the solution of the backward evolution equation
\[
\frac{d\xi_s(t)}{dt}=-A\xi_s(t)-\Psi_{\Delta t}'(X^{(\Delta t)}(t,x))\xi_s(t)-F(t)~,\quad \xi_s(T)=0.
\]
Indeed, for all $h\in H$,
\[
\langle  h,\mathcal{U}_s^\star(F)\rangle = \int_{s}^{T}\langle U(t,s)h,F(t)\rangle dt=-\int_{s}^{T}\frac{d}{dt}\bigl(\langle U(t,s)h,\xi_s(t)\rangle\bigr) dt=\langle h,\xi_s(s)\rangle,
\]
using the conditions $\xi_s(T)=0$ and $U(s,s)h=h$.

To obtain the required estimate of $|(-A)^{\frac12}\mathcal{U}_s^\star F|^2=|(-A)^{\frac12}\xi_s(s)|^2$, compute
\begin{align*}
\frac{1}{2}\frac{d}{dt}|(-A)^{\frac12}\xi_s(t)|_H^2&=|(-A)\xi_s(t)|_H^2-\langle \Psi_{\Delta t}'(X^{(\Delta t)}(t,x))\xi_s(t),(-A)\xi_s(t)\rangle\\
&-\langle F(t),(-A)\xi_s(t)\rangle\\
&\ge |(-A)\xi_s(t)|_H^2-|\Psi_{\Delta t}'(X^{(\Delta t)}(t,x))|_{E}|\xi_s(t)|_H|(-A)\xi_s(t)|_H\\
&-|F(t)|_H|(-A)\xi_s(t)|_H\\
&\ge -\frac12|\Psi_{\Delta t}'(X^{(\Delta t)}(t,x))|_{E}^2|\xi_s(t)|_H^2-\frac12|F(t)|_H^2,
\end{align*}
thanks to Young inequality. Integrating from $t=s$ to $t=T$, and using $\xi_s(T)=0$, we have
\begin{equation}\label{eq:estim-xis}
|(-A)^{\frac12}\xi_s(s)|^2 \le \underset{0\le t\le T}\sup~|\Psi_{\Delta t}'(X(t,x))|_E^2\int_{s}^{T}|\xi_s(t)|^2dt+\frac12\int_{s}^{T}|F(t)|^2_{H}dt.
\end{equation}
Moreover, using $-\Psi_{\Delta t}'(\cdot)\ge -e^{\Delta t_0}$, see Lemma~\ref{lem:rappel}, thanks to Young inequality, we have
\begin{align*}
\frac{1}{2}\frac{d}{dt}|\xi_s(t)|_H^2&=|(-A)^{1/2}\xi_s(t)|_H^2-e^{\Delta t_0}|\xi_s(t)|_H^2-\langle F(t),\xi_s(t)\rangle\\
&\ge |(-A)^{1/2}\xi_s(t)|_H^2-e^{\Delta t_0}|\xi_s(t)|_H^2-|(-A)^{-\frac12}F(t)|_H|(-A)^{\frac12}\xi_s(t)|_H\\
&\ge -e^{\Delta t_0}|\xi_{s}(t)|_H^2-\frac{1}{2}|(-A)^{-\frac12}F(t)|_H^2,
\end{align*}
then by Gronwall Lemma there exists $C(T,\Delta t_0)\in(0,\infty)$ such that for all $s\le t\le T$,
\[
|\xi_s(t)|^2\le C(T,\Delta t_0)\int_{s}^{T}|(-A)^{-1/2}F(r)|^2dr.
\]
Hence, with the inequality $|(-A)^{-\frac12}\cdot|_H\le C|\cdot|_H$ and Equation \eqref{eq:estim-xis}, we obtain Equation \eqref{eq:duality}, leading to estimate ~\eqref{eq:lemma_integral}, concluding the proof of Lemma~\ref{lem:aux-eta}.


\section{Proof of Theorem~\ref{th:weak_expo}}\label{sec:expo}

The aim of this section is to prove Theorem~\ref{th:weak_expo}. Let the numerical scheme $\bigl(X_n\bigr)_{n\in\N}$ be given by~\eqref{eq:scheme_expo}.

The section is organized as follows. An auxiliary process $\tilde{X}$ and an appropriate decomposition of the error are given in Section~\ref{sec:expo_1}. Error terms are estimated in Section~\ref{sec:expo_2}. Auxiliary results are proved in Section~\ref{sec:expo_3}.

\subsection{Decomposition of the error}\label{sec:expo_1}

As explained in Section~\ref{sec:weak}, the strategy for the weak error analysis requires to apply It\^o formula, hence the definition of an appropriate continuous-time process $\tilde{X}$.

Set, for every $n\in\N$, and every $t\in[n\Delta t,(n+1)\Delta t]$,
\begin{equation}\label{eq:scheme_expo_tilde}
\tilde{X}(t)=e^{(t-n\Delta t)A}\Phi_{t-n\Delta t}(X_n)+\int_{n\Delta t}^{t}e^{(t-s)A}dW(s).
\end{equation}
By construction, $\tilde{X}(n\Delta t)=X_n$ for all $n\in\N$. Moreover,
\[
d\tilde{X}(t)=A\tilde{X}(t)dt+dW(t)+e^{(t-n\Delta t)A}\Psi_0(\Phi_{t-n\Delta t}(X_n))dt~,\quad t\in[n\Delta t,(n+1)\Delta t], n\in\N.
\]
Recall that $\Phi_{t-n\Delta t}$ and $\Psi_0$ are defined by~\eqref{eq:PhiPsi}.

The following result gives moment estimates. Proof is postponed to Section~\ref{sec:expo_3}.
\begin{lemma}\label{lem:borneXtilde}
Let $T\in(0,\infty)$, $\Delta t_0\in(0,1)$ and $M\in\N$. There exists $C(T,\Delta t_0,M)\in(0,\infty)$ such that, for all $\Delta t\in(0,\Delta t_0]$ and $x\in E$,
\[
\underset{t\in[0,T]}\sup~\E[|\tilde{X}(t)|_E^{M}]\le C(T,\Delta t_0,M)(1+|x|_E)^M.
\]
\end{lemma}

The error is then decomposed as follows, using It\^o formula, and the Kolmogorov equation~\eqref{eq:Kolmo_Deltat}, with $T=N\Delta t$
\begin{align*}
\E\bigl[u^{(\Delta t)}(T,x)\bigr]&-\E\bigl[u^{(\Delta t)}(0,X_N)\bigr]\\
&=\sum_{k=0}^{N-1}\Bigl(\E\bigl[u^{(\Delta t)}((N-k)\Delta t,X_k)\bigr]-\E\bigl[u^{(\Delta t)}((N-k-1)\Delta t,X_{k+1})\bigr]\Bigr)\\
&=\sum_{k=0}^{N-1}\E\int_{k\Delta t}^{(k+1)\Delta t}\langle Du^{(\Delta t)}(T-t,\tilde{X}(t)),\Psi_{\Delta t}(\tilde{X}(t))-e^{(t-k\Delta t)A}\Psi_0(\Phi_{t-k\Delta t}(X_k))\rangle dt\\
&=\sum_{k=0}^{N-1}\bigl(d_k^1+d_k^2+d_k^3+d_k^4\bigr)
\end{align*}
where
\begin{align*}
d_k^1&=\E\int_{k\Delta t}^{(k+1)\Delta t}\langle Du^{(\Delta t)}(T-t,\tilde{X}(t)),\Psi_{\Delta t}(\tilde{X}(t))-\Psi_{\Delta t}(X_k)\rangle dt,\\
d_k^2&=\E\int_{k\Delta t}^{(k+1)\Delta t}\langle Du^{(\Delta t)}(T-t,\tilde{X}(t)),\bigl(I-e^{(t-k\Delta t)A}\bigr)\Psi_{\Delta t}(X_k)\rangle dt,\\
d_k^3&=\E\int_{k\Delta t}^{(k+1)\Delta t}\langle Du^{(\Delta t)}(T-t,\tilde{X}(t)),e^{(t-k\Delta t)A}\bigl(\Psi_{\Delta t}(X_k)-\Psi_{\Delta t}(\Phi_{t-k\Delta t}(X_k))\bigr)\rangle dt,\\
d_k^4&=\E\int_{k\Delta t}^{(k+1)\Delta t}\langle Du^{(\Delta t)}(T-t,\tilde{X}(t)),e^{(t-k\Delta t)A}\bigl(\Psi_{\Delta t}(\Phi_{t-k\Delta t}(X_k))-\Psi_{0}(\Phi_{t-k\Delta t}(X_k))\bigr)\rangle dt.
\end{align*}

\subsection{Estimates of error terms}\label{sec:expo_2}

\subsection*{Treatment of $d_k^1$}

Let $\eta>\frac14$, $\alpha\in(0,\frac12)$, and $\epsilon>0$, such that $\alpha+3\epsilon<\frac12$.

Assume $k=0$, then, thanks to Theorem~\ref{th:Du} and Lemmas~\ref{lem:rappel} and Lemma~\ref{lem:borneXtilde}
\[
|d_0^1|
\le C(T,\Delta t_0)\|\varphi\|_{1,\infty}\Delta t(1+|x|_E^5).
\]

Assume that $1\le k\le N-1$. Thanks to Theorem~\ref{th:Du},
\begin{align*}
|d_k^1|&\le\int_{k\Delta t}^{(k+1)\Delta t}\E\bigl[|\langle Du^{(\Delta t)}(T-t,\tilde{X}(t)),\Psi_{\Delta t}(\tilde{X}(t))-\Psi_{\Delta t}(X_k)\rangle|\bigr] dt\\
&\le \int_{k\Delta t}^{(k+1)\Delta t}\frac{C_{\eta+\frac{\alpha+\epsilon}{2}}(T,\Delta t_0)}{(T-t)^{\eta+\frac{\alpha+\epsilon}{2}}}\|\varphi\|_{1,\infty}\E\bigl[(1+|\tilde{X}(t)|_E^2)\big|(-A)^{-\eta-\frac{\alpha+\epsilon}{2}}\bigl(\Psi_{\Delta t}(\tilde{X}(t))-\Psi_{\Delta t}(X_k)\bigr)\big|_H\bigr]dt
\end{align*}

Thanks to Lemma~\ref{lem:Psi} and Taylor formula,
\begin{align*}
\big|(-A)^{-\eta-\frac{\alpha+\epsilon}{2}}\bigl(\Psi_{\Delta t}(\tilde{X}(t))-\Psi_{\Delta t}(X_k)\bigr)\big|_H&\le \mathcal{M}(X_k,\tilde{X}(t))\big|(-A)^{-\frac{\alpha}{2}}(\tilde{X}(t)-X_k)\big|_H,
\end{align*}
with
\[
\mathcal{M}(X_k,\tilde{X}(t))\le C_{\alpha,\epsilon}(\Delta t_0)(1+|\tilde{X}(t)|_E^2+|(-A)^{\frac{\alpha+2\epsilon}{2}}\tilde{X}(t)|_H^2+|X_k|_E^2+|(-A)^{\frac{\alpha+2\epsilon}{2}}X_k|_H^2).
\]

Using H\"older inequality, we need to estimate $\bigl(\E \mathcal{M}(X_k,\tilde{X}(t))^4\bigr)^{\frac14}$ and $\bigl(\E\big|(-A)^{-\frac{\alpha}{2}}(\tilde{X}(t)-X_k)\big|_H^2\bigr)^{\frac12}$. The following auxiliary results give the required estimates.
\begin{lemma}\label{lem:borneXtilde_alpha}
Let $T\in(0,\infty)$, $\Delta t_0\in(0,1)$, $\alpha\in[0,\frac{1}{2})$ and $M\in\N$. There exists $C_\alpha(T,\Delta t_0,M)\in(0,\infty)$ such that, for all $\Delta t\in(0,\Delta t_0]$, $x\in E$, and $n\in\N$ with $n\Delta t\le T$,
\[
\bigl(\E[|(-A)^{\frac{\alpha}{2}}W^A(n\Delta t)|_H^{M}]\bigr)^{\frac{1}{M}}+\bigl(\E[|(-A)^{\frac{\alpha}{2}}X_n|_H^{M}]\bigr)^{\frac{1}{M}}\le C_{\alpha}(T,\Delta t_0,M)(1+|x|_E+(n\Delta t)^{-\frac{\alpha}{2}}|x|_H).
\]
\end{lemma}

\begin{lemma}\label{lem:borneXtilde_-alpha}
Let $T\in(0,\infty)$, $\Delta t_0\in(0,1)$, $\alpha\in[0,\frac{1}{2})$. There exists $C_\alpha(T,\Delta t_0)\in(0,\infty)$ such that, for all $\Delta t\in(0,\Delta t_0]$ and $x\in E$, for all $t\in[0,T]\cap [n\Delta t,(n+1)\Delta t]$, $n\in\N$,
\[
\bigl(\E|(-A)^{-\frac{\alpha}{2}}(\tilde{X}(t)-X_n)|_H^2\bigr)^{\frac12}\le C_\alpha(T,\Delta t_0)\Delta t^\alpha (1+|x|_E^3+(n\Delta t)^{-\frac{\alpha}{2}}|x|_H).
\]
\end{lemma}

Thanks to Lemma~\ref{lem:rappel}, $\Phi_{\Delta t}$ is globally Lipschitz continuous, hence applying~\eqref{eq:inequality2} yields
\begin{align*}
|(-A)^{\frac{\alpha+2\epsilon}{2}}\tilde{X}(t)|_H&\le |(-A)^{\frac{\alpha+2\epsilon}{2}}\Phi_{t-n\Delta t}(X_n)|_H+|(-A)^{\frac{\alpha+2\epsilon}{2}}\bigl(W^A(t)-e^{(t-n\Delta t)A}W^A(n\Delta t)\bigr)|_H\\
&\le Ce^{\Delta t}|(-A)^{\frac{\alpha+3\epsilon}{2}}X_n|_H+|(-A)^{\frac{\alpha+2\epsilon}{2}}W^A(t)|+|(-A)^{\frac{\alpha+2\epsilon}{2}}W^A(n\Delta t)|_H.
\end{align*}
With the condition $\alpha+3\epsilon<\frac{1}{2}$, Lemma~\ref{lem:borneXtilde_alpha} above then implies
\[
\bigl(\E \mathcal{M}(X_k,\tilde{X}(t))^4\bigr)^{\frac14}\le C(1+|x|_E^2+(k\Delta t)^{-\alpha-3\epsilon}|x|_H^2).
\]

Finally, using Lemmas~\ref{lem:borneXtilde} and~\ref{lem:borneXtilde_-alpha}, for $1\le k\le n-1$,
\begin{align*}
|d_k^1|&
\le \Delta t^\alpha \int_{k\Delta t}^{(k+1)\Delta t}\frac{C_{\alpha,\epsilon}(T,\Delta t_0)}{(T-t)^{\eta+\frac{\alpha+\epsilon}{2}}}\|\varphi\|_{1,\infty}dt \left(1+\frac{1}{(k\Delta t)^{2\alpha+3\epsilon}}\right)(1+|x|_E^5)
\end{align*}

Observe that $2\alpha+3\epsilon<1$.

\subsection*{Treatment of $d_k^2$}

Let $\alpha\in(0,\frac12)$. Then, thanks to Theorem~\ref{th:Du}, for all $k\in\left\{0,\ldots,N-1\right\}$,
\begin{align*}
|d_k^2|&\le\int_{k\Delta t}^{(k+1)\Delta t}\E\bigl[|\langle Du^{(\Delta t)}(T-t,\tilde{X}(t)),\bigl(I-e^{(t-k\Delta t)A}\bigr)\Psi_{\Delta t}(X_k)\rangle|\bigr] dt\\
&\le \int_{k\Delta t}^{(k+1)\Delta t}\frac{C_{2\alpha}(T,\Delta t_0)\|\varphi\|_{1,\infty}}{(T-t)^{2\alpha}}\|(-A)^{-2\alpha}(I-e^{(t-k\Delta t)A})\|_{\mathcal{L}(H)}\E\bigl[(1+|\tilde{X}(t)|_E^2)|\Psi_{\Delta t}(X_k)|_H\bigr] dt\\
&\le \Delta t^{2\alpha}\int_{k\Delta t}^{(k+1)\Delta t}\frac{C_{2\alpha}(T,\Delta t_0)\|\varphi\|_{1,\infty}}{(T-t)^{2\alpha}}(1+|x|_E^5)dt,
\end{align*}
using $\|(-A)^{-2\alpha}(I-e^{(t-k\Delta t)A})\|_{\mathcal{L}(H)}\le C_{2\alpha}(t-k\Delta t)^{2\alpha}\le C_{2\alpha}\Delta t^{2\alpha}$ when $0\le t-k\Delta t\le \Delta t$.

\subsection*{Treatment of $d_k^3$}

Thanks to Lemma~\ref{lem:rappel}, for all $z\in\R$, and all $0\le \tau\le \Delta t\le \Delta t_0$,
\begin{align*}
|\Psi_{\Delta t}(z)-\Psi_{\Delta t}(\Phi_{\tau}(z))|&\le C(\Delta t_0)|z-\Phi_{\tau}(z)|(1+|z|^2)\\
&\le C(\Delta t_0)\tau |\Psi_\tau(z)|(1+|z|^2)\\
&\le C(\Delta t_0)\Delta t (1+|z|^5).
\end{align*}

Thanks to Theorem~\ref{th:Du}, for all $k\in\left\{0,\ldots,N-1\right\}$,
\begin{align*}
|d_k^3|&\le \int_{k\Delta t}^{(k+1)\Delta t}\E\bigl[|\langle Du^{(\Delta t)}(T-t,\tilde{X}(t)),e^{(t-k\Delta t)A}\bigl(\Psi_{\Delta t}(X_k)-\Psi_{\Delta t}(\Phi_{t-k\Delta t}(X_k))\bigr)\rangle|\bigr] dt\\
&\le \int_{k\Delta t}^{(k+1)\Delta t}C_0(T,\Delta t_0)\|\varphi\|_{1,\infty}\E\bigl[\big|\Psi_{\Delta t}(X_k)-\Psi_{\Delta t}(\Phi_{t-k\Delta t}(X_k))\big|_H\bigr]dt\\
&\le \Delta t\int_{k\Delta t}^{(k+1)\Delta t}C_0(T,\Delta t_0)\|\varphi\|_{1,\infty}\E\bigl[(1+|X_k|_E^5)\bigr]dt\\
&\le \Delta t^2 C(T,\Delta t_0)\|\varphi\|_{1,\infty}(1+|x|_E^5).
\end{align*}

\subsection*{Treatment of $d_k^4$}

Thanks to Theorem~\ref{th:Du} and Lemma~\ref{lem:rappel}, for all $k\in\left\{0,\ldots,N-1\right\}$,
\begin{align*}
|d_k^4|&\le \int_{k\Delta t}^{(k+1)\Delta t}\E\bigl[|\langle Du^{(\Delta t)}(T-t,\tilde{X}(t)),e^{(t-k\Delta t)A}\bigl(\Psi_{\Delta t}(\Phi_{t-k\Delta t}(X_k))-\Psi_{0}(\Phi_{t-k\Delta t}(X_k))\bigr)\rangle|\bigr] dt\\
&\le \int_{k\Delta t}^{(k+1)\Delta t}C_0(T,\Delta t_0)\|\varphi\|_{1,\infty}\E\bigl[\big|\Psi_{\Delta t}(\Phi_{t-k\Delta t}(X_k))-\Psi_{0}(\Phi_{t-k\Delta t}(X_k))\big|_H\bigr]dt\\
&\le \Delta t\int_{k\Delta t}^{(k+1)\Delta t}C_0(T,\Delta t_0)\|\varphi\|_{1,\infty}\E\bigl[(1+|X_k|_E^5)\bigr]dt\\
&\le \Delta t^2 C(T,\Delta t_0)\|\varphi\|_{1,\infty}(1+|x|_E^5).
\end{align*}

\subsection*{Conclusion}

In conclusion,
\begin{align*}
\big|\E\bigl[u^{(\Delta t)}(T,x)\bigr]&-\E\bigl[u^{(\Delta t)}(0,X_N)\bigr]\big|\le C(T,\Delta t_0,|x|_E)\|\varphi\|_{1,\infty}\\
&\Bigl(\Delta t+\Delta t^\alpha \sum_{k=1}^{N-1}\int_{k\Delta t}^{(k+1)\Delta t}\frac{C_{\alpha,\epsilon}(T,\Delta t_0)}{(T-t)^{\eta+\frac{\alpha+\epsilon}{2}}t^{2\alpha+3\epsilon}}dt\\
&\quad +\Delta t^{2\alpha}\sum_{k=0}^{N-1}\int_{k\Delta t}^{(k+1)\Delta t}\frac{C_{2\alpha}(T,\Delta t_0)\|\varphi\|_{1,\infty}}{(T-t)^{2\alpha}}dt+\Delta t (N\Delta t)\Bigr)\\
&\le C(T,\Delta t_0,|x|_E)\|\varphi\|_{1,\infty}\Delta t^\alpha,
\end{align*}
with $\alpha<\frac{1}{2}$, and $\epsilon>0$ such that $\alpha+3\epsilon<\frac{1}{2}$.

Combined with the arguments of Section~\ref{sec:weak}, this concludes the proof of Theorem~\ref{th:weak_expo}.

\subsection{Proof of Lemmas~\ref{lem:borneXtilde},~\ref{lem:borneXtilde_alpha} and~\ref{lem:borneXtilde_-alpha}}\label{sec:expo_3}

\begin{proof}[Proof of Lemma~\ref{lem:borneXtilde}]
For any $n\in\N$, and $t\in[n\Delta t,(n+1)\Delta t]$, the definition~\eqref{eq:scheme_expo_tilde} of $\tilde{X}(t)$ gives 
\begin{align*}
|\tilde{X}(t)|_E&\le |\Phi_{t-n\Delta t}(X_n)|_E+|W^A(t)-e^{(t-n\Delta t)A}W^A(n\Delta t)|_E\\
&\le e^{\Delta t}|X_n|_E+|W^A(t)|_E+|W^A(n\Delta t)|_E,
\end{align*}
thanks to Lemma~\ref{lem:rappel}. Using~\eqref{eq:borne_W^A} and Lemma~\ref{lem:moment-Xn} then concludes the proof.
\end{proof}

\begin{proof}[Proof of Lemma~\ref{lem:borneXtilde_alpha}]

Note that, for all $n\in\N$, such that $n\Delta t\le T$,
\begin{align*}
\bigl(\E|(-A)^{\frac{\alpha}{2}}X_n|_H^M|\bigr)^{\frac{1}{M}}&\le |(-A)^{\frac{\alpha}{2}}e^{n\Delta tA}x|_H+C(T)\Delta t\sum_{k=0}^{n-1}\bigl(\E|(-A)^{\frac{\alpha}{2}}e^{(n-k)\Delta t}\Psi_{\Delta t}(X_k)|_H^M\bigr)^{\frac{1}{M}}\\
&+\bigl(\E|(-A)^{\frac{\alpha}{2}}W^A(n\Delta t)|_H^M\bigr))^{\frac{1}{M}}\\
&\le (n\Delta t)^{-\frac{\alpha}{2}}|x|_E+C(T)\Delta t\sum_{k=0}^{n-1}\frac{1}{\bigl((n-k)\Delta t)^{\frac{\alpha}{2}}}\bigl(\E|\Psi_{\Delta t}(X_k)|_H^M\bigr)^{\frac{1}{M}}+C_\alpha(T,M),
\end{align*}
since, for all $t\in[0,T]$,
\begin{align*}
\E|(-A)^{\frac{\alpha}{2}}W^A(t)|_H^M&\le C\bigl(\int_{0}^{t}\|(-A)^\frac{\alpha}{2}e^{(t-s)A}\|_{\mathcal{L}_2(H)}^2ds\bigr)^{\frac{M}{2}}\\
&\le C\bigl(\int_{0}^{t}(t-s)^{-\alpha-\frac{1}{2}-\epsilon}ds\bigr)^{\frac{M}{2}}\\
&\le C_\alpha(T,M),
\end{align*}
where $\epsilon\in(0,\frac{1}{2}-\alpha)$. Thanks to Lemma~\ref{lem:borneXtilde}, then
\[
\bigl(\E|(-A)^{\frac{\alpha}{2}}X_n|_H^M|\bigr)^{\frac{1}{M}}\le C_\alpha(T,\Delta t_{0},M)(1+|x|_E+(n\Delta t)^{-\frac{\alpha}{2}}|x|_H).
\]
\end{proof}

\begin{proof}[Proof of Lemma~\ref{lem:borneXtilde_-alpha}]
For $t\in[n\Delta t,(n+1)\Delta t]$, $t\le T$,
\begin{align*}
\tilde{X}(t)-X_n=e^{(t-n\Delta t)A}X_n-X_n+(t-n\Delta t)e^{(t-n\Delta t)A}\Psi_{t-n\Delta t}(X_n)+\int_{n\Delta t}^{t}e^{(t-s)A}dW(s).
\end{align*}
First,
\begin{align*}
\bigl(\E|(-A)^{-\frac{\alpha}{2}}(e^{(t-n\Delta t)A}-I)X_n|_H^2\bigr)^{\frac12}&\le \|(-A)^{-\alpha}(e^{(t-n\Delta t)A}-I)\|_{\mathcal{L}(H)} \bigl(\E|(-A)^{\frac{\alpha}{2}}X_n|_H^2\bigr)^{\frac12}\\
&\le C_{\alpha}(T,\Delta t_0)\Delta t^\alpha (1+|x|_E+(n\Delta t)^{-\frac{\alpha}{2}}|x|_H),
\end{align*}
thanks to Lemma~\ref{lem:borneXtilde_alpha}. Second,
\begin{align*}
\bigl(\E\big|(-A)^{-\frac{\alpha}{2}}(t-n\Delta t)e^{(t-n\Delta t)A}\Psi_{t-n\Delta t}(X_n)\big|_H^2\bigr)^{\frac12}&\le \Delta t\bigl(\E|\Psi_{t-n\Delta t}(X_n)|_E^2\bigr)^{\frac12}\\
&\le C(T,\Delta t_0)\Delta t(1+|x|_E^3),
\end{align*}
thanks to Lemma~\ref{lem:borneXtilde}. Third, by It\^o formula, for $\epsilon=\frac{1}{2}-\alpha>0$,
\begin{align*}
\E|(-A)^{-\frac{\alpha}{2}}\int_{n\Delta t}^{t}e^{(t-s)A}dW(s)|_H^2&=\int_{n\Delta t}^{t}\|(-A)^{-\frac{\alpha}{2}}e^{(t-s)A}\|_{\mathcal{L}_2(H)}^2ds\\
&\le C_\alpha \int_{n\Delta t}^{t}(t-s)^{-\frac{1}{2}-\epsilon+\alpha}ds\\
&\le C_{\alpha}\Delta t^{\frac{1}{2}-\epsilon+\alpha}=C_{\alpha}\Delta t^{2\alpha}.
\end{align*}
This concludes the proof of Lemma~\ref{lem:borneXtilde_-alpha}.
\end{proof}


\section{Proof of Theorem~\ref{th:weak_implicit}}\label{sec:impl}

The aim of this section is to prove Theorem~\ref{th:weak_implicit}. Let the numerical scheme $\bigl(X_n\bigr)_{n\in\N}$ be given by~\eqref{eq:scheme_implicit}.

The section is organized as follows. An auxiliary process $\tilde{X}$ and an appropriate decomposition of the error are given in Section~\ref{sec:impl1}. Error terms are estimated in Section~\ref{sec:impl2}. Auxiliary results are proved in Section~\ref{sec:impl3}.

Assume that $\varphi$ satisfies Assumption~\ref{ass:varphi}, and to simplify notation, without loss of generality, assume that $\|\varphi\|_{2,\infty}\le 1$.

\subsection{Decomposition of the error}\label{sec:impl1}

As explained in Section~\ref{sec:weak}, the strategy for the weak error analysis requires to apply It\^o formula, hence the definition of an appropriate continuous-time process $\tilde{X}$.

Set, for every $n\in\N$, and every $t\in[n\Delta t,(n+1)\Delta t]$,
\begin{equation}\label{eq:scheme_implicit_tilde}
\tilde{X}(t)=X_n+(t-n\Delta t)AS_{\Delta t}X_n+(t-n\Delta t)S_{\Delta t}\Psi_{\Delta t}(X_n)+S_{\Delta t}\bigl(W(t)-W(n\Delta t)\bigr),
\end{equation}
where we recall that $S_{\Delta t}=(I-\Delta tA)^{-1}$.

By construction, $\tilde{X}(n\Delta t)=X_n$ for all $n\in\N$. Moreover,
\[
d\tilde{X}(t)=AS_{\Delta t}X_ndt+S_{\Delta t}\Psi_{\Delta t}(X_n)dt+S_{\Delta t}dW(t)~,\quad t\in[n\Delta t,(n+1)\Delta t],n\in\N.
\]
The following result gives moment estimates. Proof is postponed to Section~\ref{sec:impl1}.
\begin{lemma}\label{lem:borneXtilde-implicit}
Let $T\in(0,\infty)$, $\Delta t_0\in(0,1)$ and $M\in\N$. There exists $C(T,\Delta t_0,M)\in(0,\infty)$ such that, for all $\Delta t\in(0,\Delta t_0]$ and $x\in E$,
\[
\underset{t\in[0,T]}\sup~\E[|\tilde{X}(t)|_E^{M}]\le C(T,\Delta t_0,M)(1+|x|_E^3)^M.
\]
\end{lemma}

The error is then decomposed as follows, using It\^o formula, and the Kolmogorov equation~\eqref{eq:Kolmo_Deltat}, with $T=N\Delta t$,
\begin{align*}
\E[u^{(\Delta t)}(T,x)]-&\E[u^{(\Delta t)}(0,X_N)]\\
&=\E[u^{(\Delta t)}(T,x)]-\E[u^{(\Delta t)}(T-\Delta t,X_1)]\\
&+\sum_{k=1}^{N-1}\Bigl(\E\bigl[u^{(\Delta t)}((N-k)\Delta t,X_k)\bigr]-\E\bigl[u^{(\Delta t)}((N-k-1)\Delta t,X_{k+1})\bigr]\Bigr)\\
&=\E[u^{(\Delta t)}(T-\Delta t,X^{(\Delta t)}(\Delta t))]-\E[u^{(\Delta t)}(T-\Delta t,X_1)]\\
&+\sum_{k=1}^{N-1}\bigl(a_k+b_k+c_k\bigr),
\end{align*}
where
\begin{align*}
a_k&=\int_{k\Delta t}^{(k+1)\Delta t}\E \langle Du^{(\Delta t)}(T-t,\tilde{X}(t)),A\tilde{X}(t)-AS_{\Delta t}X_k\rangle dt\\
b_k&=\int_{k\Delta t}^{(k+1)\Delta t}\E \langle Du^{(\Delta t)}(T-t,\tilde{X}(t)),\Psi_{\Delta t}(\tilde{X}(t))-S_{\Delta t}\Psi_{\Delta t}(X_k)\rangle dt\\
c_k&=\frac{1}{2}\int_{k\Delta t}^{(k+1)\Delta t}\E\left[\sum_{j\in\N}D^2u^{(\Delta t)}(T-t,\tilde{X}(t)).(e_j,e_j)\left(1-\frac{1}{(1+\lambda_j\Delta t)^2}\right)\right]dt.
\end{align*}

Section~\ref{sec:impl2} is devoted to the proof of Lemmas~\ref{lem:term0} and~\ref{lem:akbkck} below. Theorem~\ref{th:weak_implicit} is a straightforward consequence of these results, thanks to the decomposition of the error above.
\begin{lemma}\label{lem:term0}
Let $T\in(0,\infty)$, $\Delta t_0\in(0,1]$ and $x\in E$. For all $\alpha\in[0,\frac12)$, there exists $C_{\alpha}(T,\Delta t_0,|x|_E)\in(0,\infty)$ such that, for all $\Delta t\in(0,\Delta t_0)$,
\[
\big|\E[u^{(\Delta t)}(T-\Delta t,X^{(\Delta t)}(\Delta t))]-\E[u^{(\Delta t)}(T-\Delta t,X_1)]\big|\le C_\alpha(T,\Delta t_{0},|x|_E)\Delta t^{\alpha}.
\]
\end{lemma}
\begin{lemma}\label{lem:akbkck}
Let $T\in(0,\infty)$, $\Delta t_0\in(0,1]$ and $x\in E$. For all $\alpha\in[0,\frac12)$, there exists $C_{\alpha}(T,\Delta t_0,|x|_E)\in(0,\infty)$ such that, for all $\Delta t\in(0,\Delta t_0)$,
\[
\sum_{k=1}^{n-1}(|a_k|+|b_k|+|c_k|)\le C_{\alpha}(T,\Delta t_{0},|x|_E)\Delta t^\alpha.
\]
\end{lemma}

Compared with Section~\ref{sec:expo}, there are more error tems, and the analysis is more technical. The proof essentially follows the same strategy as in~\cite{Debussche:11}. In particular, Malliavin calculus techniques are employed. For completeness, details are given.

We emphasize that the most important new result is the estimate on the Malliavin derivative, see Lemma~\ref{lem:malliavin}. This result is non trivial, since $\Psi_{\Delta t}$ is not globally Lipschitz continuous, and it is obtained thanks to the structure of the numerical scheme, based on a splitting approach.

Note that an alternative approach to treat the term $b_k$ below, would be to adapt the strategy used in Section~\ref{sec:expo_2} to treat the term $d_k^1$, using Lemma~\ref{lem:Psi} in particular. Appropriate versions of Lemmas~\ref{lem:borneXtilde_alpha} and~\ref{lem:borneXtilde_-alpha} would be required. It seems that this alternative approach does not considerably shortens the proof.

We recall the Malliavin calculus duality formula in some Hilbert space $K$. Let $\mathbb{D}^{1,2}$ be the closure of smooth random variables (with respect to Malliavin derivative) for the topology defined by the norm
\[
\|F\|_{\mathbb{D}^{1,2}} =\left(\E[|F|_{K}^{2}] + \E\left[\int_{0}^{T}|\mathcal{D}_{s}F|_{K}^{2} ds\right]\right).
\]
where $\mathcal{D}_sF$ denotes the Malliavin derivative of $F$.
For $F \in \mathbb{D}^{1,2}$ and $\Xi\in L^{2}(\Omega\times[0,T]; K)$ such that $\Xi(t) \in \mathbb{D}^{1,2}$ for all $t\in[0,T]$ and $\int_{0}^{T}\int_{0}^{T}|\mathcal{D}_{s}\Xi(t)|^{2} ds dt < +\infty$, we have the integration by part formula:
\[
\E\left[F \int_{0}^{T} ( \Xi(s) ,dW_{s})\right] = \E\left[\int_{0}^{T}\langle \mathcal{D}_{s}F, \Xi(s)\rangle ds\right].
\]
In our context, we will use another form of this integration by parts formula: for $u\in\mathcal{C}_b^2(H)$, and any adapted process $\Xi \in L^{2}(\Omega\times[0,T]; \mathcal{L}_{2}(H))$,
\begin{equation}\label{eq:malliavin-duality}
\E\left[\langle Du^{(\Delta t)}(F), \int_{0}^{T} \Xi(s) dW(s)\rangle \right] = \E\left[\sum_{j\in\N}\int_{0}^{T} D^{2}u^{(\Delta t)}(F). ( \mathcal{D}_s^{e_j}F,\Xi(s) e_{j}) ds\right].
\end{equation}

\subsection{Estimates of error terms}\label{sec:impl2}

\subsubsection{Proof of Lemma~\ref{lem:term0}}

For any $\alpha\in[0,\frac12)$, thanks to Theorem~\ref{th:Du},
\begin{align*}
\big|\E[u^{(\Delta t)}(T-\Delta t,X^{(\Delta t)}(\Delta t))]-&\E[u^{(\Delta t)}(T-\Delta t,X_1)]\big|\\
&\le \frac{C_\alpha}{(T-\Delta t)^{2\alpha}}\E|(-A)^{-2\alpha}(X^{(\Delta t)}(\Delta t)-X_1)|_H.
\end{align*}

Note that
\begin{align*}
\E|(-A)^{-2\alpha}&(X^{(\Delta t)}(\Delta t)-X_1)|_H\le |(-A)^{-2\alpha}\bigl(e^{\Delta t A}-S_{\Delta t}\bigr)x|_H\\
&+\int_{0}^{\Delta t}\E|(-A)^{-2\alpha}e^{(\Delta t-t)A}\Psi_{\Delta t}(X^{(\Delta t)}(t))|_Hdt+\Delta t |(-A)^{-2\alpha}S_{\Delta t}\Psi_{\Delta t}(x)|_H\\
&+\E|\int_{0}^{\Delta t}(-A)^{-2\alpha}e^{(\Delta t-t)A}dW(t)|_H+\E|\int_{0}^{\Delta t}(-A)^{-2\alpha}S_{\Delta t}dW(t)|_H.
\end{align*}

First,
\begin{align*}
|(-A)^{-2\alpha}\bigl(e^{\Delta t A}-S_{\Delta t}\bigr)x|_H&\le \bigl(\|(-A)^{-2\alpha}(e^{\Delta tA}-I)\|_{\mathcal{L}(H)}+\|(-A)^{-2\alpha}(S_{\Delta t}-I)\|_{\mathcal{L}(H)}\bigr)|x|_H\\
&\le C_\alpha \Delta t^{2\alpha}|x|_H.
\end{align*}
Second, since $\|(-A)^{-2\alpha}\|_{\mathcal{L}(H)}<\infty$ for $\alpha\ge 0$, then using Lemma~\ref{lem:moment-Xdeltat} gives
\[
\E|(-A)^{-\alpha}e^{(\Delta t-t)A}\Psi_{\Delta t}(X^{(\Delta t)}(t))|_H\le \E|\Psi_{\Delta t}(X^{(\Delta t)}(t))|_H\le C(1+|x|_E^3)
\]
and
\[
|(-A)^{-\alpha}S_{\Delta t}\Psi_{\Delta t}(x)|_H\le |\Psi_{\Delta t}(x)|_H\le C(1+|x|_E^3).
\]
Finally, for the stochastic integral terms,
\[
\E|\int_{0}^{\Delta t}(-A)^{-2\alpha}e^{(\Delta t-t)A}dW(t)|_H^2+\E|\int_{0}^{\Delta t}(-A)^{-2\alpha}S_{\Delta t}dW(t)|_H^2\le 2\Delta t\|(-A)^{-2\alpha}\|_{\mathcal{L}_2(H)}^2
\]
and $\|(-A)^{-2\alpha}\|_{\mathcal{L}_2(H)}^2<\infty$ when $2\alpha>\frac14$.

It is then straightforward to conclude that, for $\alpha\in[0,\frac12)$,
\[
\big|\E[u^{(\Delta t)}(T-\Delta t,X^{(\Delta t)}(\Delta t))]-\E[u^{(\Delta t)}(T-\Delta t,X_1)]\big|\le C_\alpha(T,|x|_E)\Delta t^\alpha.
\]

This concludes the proof of Lemma~\ref{lem:term0}.

\subsubsection{Proof of Lemma~\ref{lem:akbkck}, Part 1}\label{sec:lem_ak}

The aim of this section is to prove, for $\alpha\in[0,\frac12)$, that
\[
\sum_{k=1}^{N-1}|a_k|\le C_{\alpha}(T,\Delta t_{0},|x|_E)\Delta t^\alpha.
\]

For that purpose, decompose $a_k$ as follows:
\[
a_k=a_k^1+a_k^2,
\]
where
\begin{align*}
a_k^1&=\int_{k\Delta t}^{(k+1)\Delta t}\E\bigl[\langle Du^{(\Delta t)}(T-t,\tilde{X}(t)),A(I-S_{\Delta t})X_k\rangle\bigr]dt,\\
a_k^2&=\int_{k\Delta t}^{(k+1)\Delta t}\E\bigl[\langle Du^{(\Delta t)}(T-t,\tilde{X}(t)),A(\tilde{X}(t)-X_k)\rangle\bigr]dt.
\end{align*}
Using the formulation
\[
X_k=S_{\Delta t}^k x+\Delta t\sum_{\ell=0}^{k-1}S_{\Delta t}^{k-\ell}\Psi_{\Delta t}(X_\ell)+\sum_{\ell=0}^{k-1}\int_{\ell\Delta t}^{(\ell+1)\Delta t}S_{\Delta t}^{k-\ell}dW(t),
\]
and the identity $I-S_{\Delta t}=-\Delta tS_{\Delta t}A$, the expression $a_k^1$ is decomposed as
\[
a_k^1=a_k^{1,1}+a_k^{1,2}+a_k^{1,3},
\]
where
\begin{align*}
a_k^{1,1}&=-\Delta t\int_{k\Delta t}^{(k+1)\Delta t}\E\bigl[ \langle Du^{(\Delta t)}(T-t,\tilde{X}(t)),A^2 S_{\Delta t}^{k+1} x\rangle\bigr] dt,\\
a_k^{1,2}&=-\Delta t\int_{k\Delta t}^{(k+1)\Delta t}\E\bigl[ \langle Du^{(\Delta t)}(T-t,\tilde{X}(t)),\Delta t\sum_{\ell=0}^{k-1}A^2 S_{\Delta t}^{k-\ell+1}\Psi_{\Delta t}(X_\ell)\rangle\bigr] dt,\\
a_k^{1,3}&=-\Delta t\int_{k\Delta t}^{(k+1)\Delta t}\E\bigl[ \langle Du^{(\Delta t)}(T-t,\tilde{X}(t)),\sum_{\ell=0}^{k-1}\int_{\ell\Delta t}^{(\ell+1)\Delta t}A^2S_{\Delta t}^{k-\ell+1}dW(t)\rangle\bigr] dt.
\end{align*}
Using~\eqref{eq:scheme_implicit_tilde}, the expression $a_k^2$ is decomposed as
\[
a_k^2=a_k^{2,1}+a_k^{2,2}+a_k^{2,3},
\]
where
\begin{align*}
a_k^{2,1}&=\int_{k\Delta t}^{(k+1)\Delta t}\E\bigl[\langle Du^{(\Delta t)}(T-t,\tilde{X}(t)),(t-t_k)A^2S_{\Delta t}X_k\rangle\bigr] dt,\\
a_k^{2,2}&=\int_{k\Delta t}^{(k+1)\Delta t}\E\bigl[\langle Du^{(\Delta t)}(T-t,\tilde{X}(t)),(t-t_k)AS_{\Delta t}\Psi_{\Delta t}(X_k)\rangle\bigr] dt,\\
a_k^{2,3}&=\int_{k\Delta t}^{(k+1)\Delta t}\E\bigl[\langle Du^{(\Delta t)}(T-t,\tilde{X}(t)),\int_{k\Delta t}^{t}AS_{\Delta t}dW(s) \rangle\bigr] dt.
\end{align*}

\subsubsection*{Treatment of $a_k^{1,1}$}

Let $\alpha\in[0,\frac12)$ and $\epsilon\in(0,1-2\alpha)$. Thanks to Theorem~\ref{th:Du} and Lemma~\ref{lem:borneXtilde-implicit},
\begin{align*}
|a_k^{1,1}|&\le C_{2\alpha+\epsilon}(T,\Delta t_{0},|x|_E)\Delta t\int_{k\Delta t}^{(k+1)\Delta t}\frac{1}{(T-t)^{2\alpha+\epsilon}}dt|(-A)^{2-2\alpha-\epsilon}S_{\Delta t}^{k+1}x|_H\\
&\le C_{\alpha,\epsilon}(T,\Delta t_{0},|x|_E)\Delta t\int_{k\Delta t}^{(k+1)\Delta t}\frac{1}{(T-t)^{2\alpha+\epsilon}}dt\|(-A)^{1-\epsilon}S_{\Delta t}^{k}\|_{\mathcal{L}(H)}\|(-A)^{1-2\alpha}S_{\Delta t}\|_{\mathcal{L}(H)}\\
&\le C_{\alpha,\epsilon}(T,\Delta t_{0},|x|_E)\Delta t^{2\alpha}\int_{k\Delta t}^{(k+1)\Delta t}\frac{1}{(k\Delta t)^{1-\epsilon}(T-t)^{2\alpha+\epsilon}}dt,
\end{align*}
thanks to the standard inequalities $\|(-A)^{1-\epsilon}S_{\Delta t}^{k}\|_{\mathcal{L}(H)}\le C_\epsilon (k\Delta t)^{-1+\epsilon}$ and $\|(-A)^{1-2\alpha}S_{\Delta t}\|_{\mathcal{L}(H)}\le C_{\alpha}\Delta t^{2\alpha-1}$.

\subsubsection*{Treatment of $a_k^{1,2}$}

Let $\alpha\in[0,\frac12)$ and $\epsilon\in(0,1-2\alpha)$. Thanks to Theorem~\ref{th:Du}, Lemma~\ref{lem:borneXtilde-implicit}, and Cauchy-Schwarz inequality,
\[
|a_k^{1,2}|\le C_{2\alpha+\epsilon}(T,\Delta t_{0},|x|_E)\Delta t\int_{k\Delta t}^{(k+1)\Delta t}\frac{1}{(T-t)^{2\alpha+\epsilon}}dt \Delta t\sum_{\ell=0}^{k-1}\bigl(\E\big|(-A)^{2-2\alpha-\epsilon}S_{\Delta t}^{k-\ell+1}\Psi_{\Delta t}(X_\ell)\big|_H^2\bigr)^{\frac12} dt.
\]

Thanks to Lemma~\ref{lem:borneXtilde-implicit},
\begin{align*}
\Delta t\sum_{\ell=0}^{k-1}\bigl(\E\big|(-A)^{2-2\alpha-\epsilon}&S_{\Delta t}^{k-\ell+1}\Psi_{\Delta t}(X_\ell)\big|_H^2\bigr)^{\frac12}\\
&\le \Delta t\sum_{\ell=0}^{k-1}\|(-A)^{1-\epsilon}S_{\Delta t}^{k-\ell}\|_{\mathcal{L}(H)}
\|(-A)^{1-2\alpha}S_{\Delta t}\|_{\mathcal{L}(H)}\bigl(\E|\Psi_{\Delta t}(X_\ell)|_H^2\bigr)^{\frac12}\\
&\le C_{\alpha,\epsilon}(T,\Delta t_{0}, |x|_E)\Delta t^{2\alpha-1},
\end{align*}
using $\Delta t\sum_{\ell=0}^{k-1}\frac{1}{((k-\ell)\Delta t)^{1-\epsilon}}\le C_\epsilon<\infty$. Thus
\[
|a_k^{1,2}|\le C_{\alpha,\epsilon}(T,\Delta t_{0},|x|_E)\int_{k\Delta t}^{(k+1)\Delta t}\frac{1}{(T-t)^{2\alpha+\epsilon}}dt\ \Delta t ^{2\alpha}.
\]

\subsubsection*{Treatment of $a_k^{1,3}$}

The Malliavin calculus duality formula~\eqref{eq:malliavin-duality} is applied, for fixed $t$, with $u=u^{(\Delta t)}(T-t,\cdot)$, $F=\tilde{X}(t)$, and $\Xi(s) = A^2S_{\Delta t}^{k-\ell+1}$ for $\ell\Delta t\le s\le (\ell+1)\Delta t$. This yields the following alternative expression for $a_{k}^{1,3}$:
\begin{align*}
a_k^{1,3}&=-\Delta t\int_{k\Delta t}^{(k+1)\Delta t}\E\bigl[ \langle Du^{(\Delta t)}(T-t,\tilde{X}(t)),\sum_{\ell=0}^{k-1}\int_{\ell\Delta t}^{(\ell+1)\Delta t}A^2S_{\Delta t}^{k-\ell+1}dW(s)\rangle\bigr] dt\\
&=-\Delta t\int_{k\Delta t}^{(k+1)\Delta t}\sum_{\ell=0}^{k-1}\int_{\ell\Delta t}^{(\ell+1)\Delta t}\sum_{j\in\N}\E\bigl[D^2u^{(\Delta t)}(T-t,\tilde{X}(t)).(\mathcal{D}_s^{e_j}\tilde{X}(t),A^2S_{\Delta t}^{k-\ell+1}e_j)\bigr]ds dt.
\end{align*}
Lemma~\ref{lem:malliavin} below provides the required estimate. Its proof is postponed to Section~\ref{sec:impl3}.

\begin{lemma}\label{lem:malliavin}
Let $T\in(0,\infty)$.

For all $k\in\N$, such that $k\Delta t\le T$, and all $s\in[0,T]$, almost surely,
\[
\|\mathcal{D}_sX_k\|_{\mathcal{L}(H)}\le e^T.
\]
In addition, $\mathcal{D}_sX_k=0$ if $k\Delta t\le s$.

Moreover, for all $0\le s<k\Delta t\le t\le (k+1)\Delta t\le T$,
\[
\|\mathcal{D}_s\tilde{X}(t)\|_{\mathcal{L}(H)}\le (3+\Delta t|\Psi_{\Delta t}'(X_k)|_E)e^T.
\]
\end{lemma}

Let $\alpha\in[0,\frac12)$, and let $\kappa\in(\frac12-\alpha,1-2\alpha)$ and $\epsilon\in(0,1-2\alpha-\kappa)$ be two auxiliary parameters. Then $2\alpha+\kappa+\epsilon<1$, and $\alpha+\kappa>\frac12$. 
Thanks to Theorem~\ref{th:D2u}, Lemma~\ref{lem:malliavin}, and the moment estimates, by Lemma~\ref{lem:borneXtilde-implicit},
\begin{align*}
|a_k^{1,3}|&\le C_{0,2\alpha+\epsilon+\kappa}(T,\Delta t_0,|x|_E)\Delta t\int_{k\Delta t}^{(k+1)\Delta t}\frac{1}{(T-t)^{2\alpha+\epsilon+\kappa}}dt\sum_{\ell=0}^{k-1}\Delta t\sum_{j\in\N}|(-A)^{2-2\alpha-\epsilon-\kappa}S_{\Delta t}^{k-\ell+1}e_j|_H\\
&\le C_{\alpha,\epsilon,\kappa}(T,\Delta t_0,|x|_E)\Delta t\int_{k\Delta t}^{(k+1)\Delta t}\frac{1}{(T-t)^{2\alpha+\epsilon+\kappa}}dt\sum_{j\in\N}\frac{\lambda_j^{1-2\alpha-\kappa}}{(1+\lambda_j\Delta t)}\\
&\le C_{\alpha,\epsilon,\kappa}(T,\Delta t_0,|x|_E)\Delta t\int_{k\Delta t}^{(k+1)\Delta t}\frac{1}{(T-t)^{2\alpha+\epsilon+\kappa}}dt\sum_{j\in\N}\frac{(\Delta t\lambda_j)^{1-\alpha}}{(1+\lambda_j\Delta t)}\frac{\Delta t^{\alpha-1}}{\lambda_j^{\alpha+\kappa}}\\
&\le C_{\alpha,\epsilon}(T,\Delta t_0,|x|_E)\int_{k\Delta t}^{(k+1)\Delta t}\frac{1}{(T-t)^{2\alpha+\epsilon+\kappa}}dt\ \Delta t^{\alpha},
\end{align*}
using $\Delta t\sum_{\ell=0}^{k-1}\frac{1}{\bigl((k-\ell)\Delta t\bigr)^{1-\epsilon}}\le C_\epsilon<\infty$, and $\sum_{j\in\N}\frac{1}{\lambda_j^{\alpha+\kappa}}<\infty$.

\subsubsection*{Treatment of $a_k^{2,1}$}

Note that $(t-t_k)AS_{\Delta t}=\frac{(t-t_k)}{\Delta t}(S_{\Delta t}-I)$. As a consequence, it is sufficient to repeat the treatment of $a_k^1$ above, and to use $t-t_k\le \Delta t$, to get the required estimate for $a_k^{2,1}$:
\[
|a_k^{2,1}|\le C_{\alpha,\epsilon}(T,\Delta t_{0},|x|_E)\Delta t^{\alpha}\int_{k\Delta t}^{(k+1)\Delta t}\frac{1}{(k\Delta t)^{1-\epsilon}(T-t)^{2\alpha+\epsilon+\kappa}}dt.
\]

\subsubsection*{Treatment of $a_k^{2,2}$}

Let $\alpha\in[0,\frac12)$. Thanks to Theorem~\ref{th:Du} and Lemma~\ref{lem:borneXtilde-implicit},
\begin{align*}
|a_k^{2,2}|&\le C_{2\alpha}(T,\Delta t_0,|x|_E)\int_{k\Delta t}^{(k+1)\Delta t}\frac{|t-t_k|}{(T-t)^{2\alpha}}dt \bigl(\E|(-A)^{1-2\alpha}S_{\Delta t}\Psi_{\Delta t}(X_k)|_H^2\bigr)^{\frac12}\\
&\le C_{2\alpha}(T,\Delta t_0,|x|_E)\int_{k\Delta t}^{(k+1)\Delta t}\frac{1}{(T-t)^{2\alpha}}dt\  \Delta t^{2\alpha}.
\end{align*}

\subsubsection*{Treatment of $a_k^{2,3}$}

Using the Malliavin calculus duality formula~\eqref{eq:malliavin-duality},
\begin{align*}
a_k^{2,3}&=\int_{k\Delta t}^{(k+1)\Delta t}\E\bigl[\langle Du^{(\Delta t)}(T-t,\tilde{X}(t)),\int_{k\Delta t}^{t}AS_{\Delta t}dW(s) \rangle\bigr] dt\\
&=\int_{k\Delta t}^{(k+1)\Delta t}\int_{k\Delta t}^{t}\sum_{j\in\N}\E\bigl[D^2u^{(\Delta t)}(T-t,\tilde{X}(t)).\bigl(\mathcal{D}_s^{e_j}\tilde{X}(t),AS_{\Delta t}e_j\bigr)\bigr]dsdt.
\end{align*}
Observe that $\mathcal{D}_s\tilde{X}(t)=S_{\Delta t}$ for $k\Delta t\le s\le t\le (k+1)\Delta t$. Let $\alpha\in[0,\frac12)$, and let $\kappa\in(\frac12-\alpha,1-2\alpha)$ be an auxiliary parameter. 
Then $2\alpha+\kappa<1$ and $\alpha+\kappa>\frac12$.
Thanks to Theorem~\ref{th:D2u} and Lemma~\ref{lem:borneXtilde-implicit},
\begin{align*}
|a_k^{2,3}|&\le \int_{k\Delta t}^{(k+1)\Delta t}\int_{k\Delta t}^{t}\E \sum_{j=1}^{\infty}\frac{\lambda_j}{(1+\lambda_j\Delta t)^2}|D^2u^{(\Delta t)}(T-t,\tilde{X}(t)).(e_j,e_j)|ds dt\\
&\le C_{0,2\alpha+\kappa}(T,\Delta t_{0},|x|_E) \sum_{j=1}^{\infty}\frac{\Delta t\lambda_j}{(1+\lambda_j\Delta t)^2\lambda_j^{2\alpha+\kappa}}\int_{k\Delta t}^{(k+1)\Delta t}\frac{1}{(T-t)^{2\alpha+\kappa}}dt\\
&\le C_{\alpha,\epsilon}(T,\Delta t_{0},|x|_E) \int_{k\Delta t}^{(k+1)\Delta t}\frac{1}{(T-t)^{2\alpha+\kappa}}dt \sum_{j\in\N}\frac{(\lambda_j\Delta t)^{1-\alpha}}{(1+\lambda_j\Delta t)^2}\frac{\Delta t^{\alpha}}{\lambda_j^{\alpha+\kappa}}\\
&\le C_{\alpha,\kappa}(T,\Delta t_{0},|x|_E) \int_{k\Delta t}^{(k+1)\Delta t}\frac{1}{(T-t)^{2\alpha+\kappa}}dt \Delta t^{\alpha},
\end{align*}
using $\sum_{j\in\N}\frac{1}{\lambda_j^{\alpha+\kappa}}<\infty$.

\subsubsection*{Conclusion}

Gathering estimates above, for all $\alpha\in[0,\frac12)$,
\begin{align*}
\sum_{k=1}^{N-1}|a_k|\le C_\alpha(T,\Delta t_0,|x|_E)\Delta t^\alpha \sum_{k=1}^{N-1}\int_{k\Delta t}^{(k+1)\Delta t}\frac{1}{t^{\beta_1(\alpha)}(T-t)^{\beta_2(\alpha)}}dt,
\end{align*}
with two parameters $\beta_1(\alpha),\beta_2(\alpha)\in[0,1)$. Therefore
\[
\sum_{k=1}^{N-1}\int_{k\Delta t}^{(k+1)\Delta t}\frac{1}{t^{\beta_1(\alpha)}(T-t)^{\beta_2(\alpha)}}dt\le \int_{0}^{T}\frac{1}{t^{\beta_1(\alpha)}(T-t)^{\beta_2(\alpha)}}dt<\infty.
\]

This concludes the first part of the proof of Lemma~\ref{lem:akbkck}.

\subsubsection{Proof of Lemma~\ref{lem:akbkck}, Part 2}\label{sec:lem_bk}

The aim of this section is to prove, for $\alpha\in[0,\frac12)$, that
\[
\sum_{k=1}^{N-1}|b_k|\le C_\alpha(T,\Delta t_0,|x|_E)\Delta t^\alpha.
\]

For that purpose, decompose $b_k$ as follows:
\[
b_k=b_k^1+b_k^2,
\]
with
\begin{align*}
b_k^1&=\int_{k\Delta t}^{(k+1)\Delta t}\E\bigl[\langle Du^{(\Delta t)}(T-t,\tilde{X}(t)),(I-S_{\Delta t})\Psi_{\Delta t}(X_k)\rangle\bigr] dt,\\
b_k^2&=\int_{k\Delta t}^{(k+1)\Delta t}\E\bigl[ \langle Du^{(\Delta t)}(T-t,\tilde{X}(t)),\Psi_{\Delta t}(\tilde{X}(t))-\Psi_{\Delta t}(X_k)\rangle\bigr] dt.
\end{align*}

Introduce real-valued functions $\Psi_{\Delta t}^j(\cdot)=\langle \Psi_{\Delta t}(\cdot),e_j\rangle$, for all $j\in\N$. Using It\^o formula, for $t\in[k\Delta t,(k+1)\Delta t]$,
\begin{align*}
\Psi_{\Delta t}^j(\tilde{X}(t))-\Psi_{\Delta t}^j(X_k)&=\int_{k\Delta t}^{t}\frac{1}{2}\sum_{i\in\N}D^2\Psi_{\Delta t}^{j}(S_{\Delta t}e_i,S_{\Delta t}e_i)ds\\
&+\int_{k\Delta t}^{t}\langle D\Psi_{\Delta t}^j(\tilde{X}(s)),S_{\Delta t}AX_k\rangle ds+\int_{k\Delta t}^{t}\langle D\Psi_{\Delta t}^{j}(\tilde{X}(s)),S_{\Delta t}\Psi_{\Delta t}(X_k)\rangle ds\\
&+\int_{k\Delta t}^{t}\langle D\Psi_{\Delta t}^{j}(\tilde{X}(s)),S_{\Delta t}dW(s)\rangle.
\end{align*}
This expansion gives the decomposition
\[
b_k^2=b_k^{2,1}+b_k^{2,2}+b_k^{2,3}+b_k^{2,4},
\]
with
\begin{align*}
b_k^{2,1}&=\int_{k\Delta t}^{(k+1)\Delta t}\int_{k\Delta t}^{t}\sum_{i\in\N}\frac{1}{(1+\lambda_i\Delta t)^2}\E\bigl[\langle Du^{(\Delta t)}(T-t,\tilde{X}(t)),D^2\Psi_{\Delta t}(\tilde{X}(s)).(e_i,e_i)\rangle \bigr] ds dt,\\
b_k^{2,2}&=\int_{k\Delta t}^{(k+1)\Delta t}\int_{k\Delta t}^{t}\E\bigl[\langle Du^{(\Delta t)}(T-t,\tilde{X}(t)),D\Psi_{\Delta t}(\tilde{X}(s)).(S_{\Delta t}AX_k)\rangle\bigr]dsdt,\\
b_k^{2,3}&=\int_{k\Delta t}^{(k+1)\Delta t}\int_{k\Delta t}^{t}\E\bigl[ \langle Du^{(\Delta t)}(T-t,\tilde{X}(t)),D\Psi_{\Delta t}(\tilde{X}(s)).(S_{\Delta t}\Psi_{\Delta t}(X_k))\rangle\bigr] ds dt,\\
b_k^{2,4}&=\int_{k\Delta t}^{(k+1)\Delta t}\E\bigl[\sum_{j=1}^{\infty}\langle Du^{(\Delta t)}(T-t,\tilde{X}(t)),e_j\rangle \int_{k\Delta t}^{t}\langle D\Psi_{\Delta t}^{j}(\tilde{X}(s)),S_{\Delta t}dW(s)\rangle \bigr] dt.
\end{align*}

\subsubsection*{Treatment of $b_k^1$}

Thanks to Theorem~\ref{th:Du} and Lemma~\ref{lem:borneXtilde-implicit},
\begin{align*}
|b_k^{1}|&\le C_{2\alpha}(T,\Delta t_{0}, x|_E)\int_{k\Delta t}^{(k+1)\Delta t}\frac{1}{(T-t)^{2\alpha}}dt\bigl(\E|(-A)^{-2\alpha}(I-S_{\Delta t})\Psi_{\Delta t}(X_k)|_H^2\bigr)^{\frac12}\\
&\le C_{2\alpha}(T,\Delta t_{0},|x|_E)\int_{k\Delta t}^{(k+1)\Delta t}\frac{1}{(T-t)^{2\alpha}}dt\ \Delta t^{2\alpha}.
\end{align*}

\subsubsection*{Treatment of $b_k^{2,1}$}

Let $\alpha\in[0,\frac12)$. Thanks to Theorem~\ref{th:Du} and Lemma~\ref{lem:borneXtilde-implicit}, and to the inequality $|D^2\Psi_{\Delta t}(x).(e_i,e_i)|_H\le C(1+|x|_E)^M|e_i|_E^2$,
\begin{align*}
|b_{k}^{2,1}|&\le C_0(T,\Delta t_0,|x|_E)\sum_{i\in\N}\frac{\Delta t^2}{(1+\Delta t\lambda_i)^2}\le C_\alpha(T,\Delta t_0,|x|_E)\sum_{i\in\N}\frac{1}{\lambda_i^{1-\alpha}}\ \Delta t^{1+\alpha}.
\end{align*}

\subsubsection*{Treatment of $b_k^{2,2}$}

A decomposition $b_k^{2,2}=b_k^{2,2,1}+b_k^{2,2,2}+b_k^{2,2,3}$ into three terms is required:
\begin{align*}
b_k^{2,2,1}&=\int_{k\Delta t}^{(k+1)\Delta t}\int_{k\Delta t}^{t}\E\bigl[\langle Du^{(\Delta t)}(T-t,\tilde{X}(t)),D\Psi_{\Delta t}(\tilde{X}(s)).(AS_{\Delta t}^{k+1}x)\rangle\bigr]dsdt,\\
b_k^{2,2,2}&=\int_{k\Delta t}^{(k+1)\Delta t}\int_{k\Delta t}^{t}\E\bigl[\langle Du^{(\Delta t)}(T-t,\tilde{X}(t)),D\Psi_{\Delta t}(\tilde{X}(s))\bigl(\Delta t\sum_{\ell=0}^{k-1}AS_{\Delta t}^{k-\ell+1}\Psi_{\Delta t}(X_\ell)\bigr)\rangle \bigr]dsdt,\\
b_k^{2,2,3}&=\int_{k\Delta t}^{(k+1)\Delta t}\int_{k\Delta t}^{t}\E\bigl[\langle Du^{(\Delta t)}(T-t,\tilde{X}(t)),D\Psi_{\Delta t}(\tilde{X}(s)).\bigl(\sum_{\ell=0}^{k-1}\int_{\ell\Delta t}^{(\ell+1)\Delta t}AS_{\Delta t}^{k-\ell+1}dW(r)\bigr)\rangle\bigr]dsdt.
\end{align*}

Let $\alpha\in[0,\frac12)$. Thanks to Theorem~\ref{th:Du}, Lemma~\ref{lem:borneXtilde-implicit}, and to inequalities already used above, the terms $b_k^{2,2,1}$ and $b_k^{2,2,2,}$ are treated as follows. First,
\begin{align*}
|b_k^{2,2,1}|&\le C_0(T,\Delta t_0,|x|_E)\Delta t^2 \|A^{2\alpha}S_{\Delta t}^{k}\|_{\mathcal{L}(H)}\|A^{1-2\alpha}S_{\Delta t}\|_{\mathcal{L}(H)}|x|_H\\
&\le C_\alpha(T,\Delta t_{0}, |x|_E)\Delta t^{1+2\alpha}\frac{1}{(k\Delta t)^{2\alpha}}.
\end{align*}

Second,
\begin{align*}
|b_k^{2,2,2}|&\le C_0(T,\Delta t_0,|x|_E)\Delta t^2 \bigl(\E|\Delta t\sum_{\ell=0}^{k-1}AS_{\Delta t}^{k-\ell+1}\Psi_{\Delta t}(X_\ell)|_H^2\bigr)^{\frac12} dsdt\\
&\le C_0(T,\Delta t_0,|x|_E)\Delta t^2 \Delta t\sum_{\ell=0}^{k-1}\|A^{2\alpha}S_{\Delta t}^{k}\|_{\mathcal{L}(H)}\|A^{1-2\alpha}S_{\Delta t}\|_{\mathcal{L}(H)}\\
&\le C(T,\Delta t_0,|x|_E)\Delta t^{1+2\alpha},
\end{align*}
using that $\Delta t\sum_{\ell=0}^{k-1}\frac{1}{((k-\ell)\Delta t)^{2\alpha}}\le C_\alpha(T)<\infty$ for $\alpha\in[0,\frac12)$.

It remains to treat $b_{k}^{2,2,3}$. Using the Malliavin calculus duality formula~\eqref{eq:malliavin-duality} and the chain rule,
\begin{align*}
b_k^{2,2,3}&=\int_{t_k}^{t_{k+1}}\int_{t_k}^{t}\sum_{\ell=0}^{k-1}\int_{t_\ell}^{t_{\ell+1}}\sum_{j\in\N}\E\bigl[D^2u^{(\Delta t)}(T-t,\tilde{X}(t)).\bigl(D\Psi_{\Delta t}(\tilde{X}(s)).(AS_{\Delta t}^{k-\ell+1}e_j),\mathcal{D}_r^{e_j}\tilde{X}(t)\bigr)\bigr]drdsdt \\
&+\int_{t_k}^{t_{k+1}}\int_{t_k}^{t}\sum_{\ell=0}^{k-1}\int_{t_\ell}^{t_{\ell+1}}\sum_{j\in\N}\E\bigl[\langle Du^{(\Delta t)}(T-t,\tilde{X}(t)),D^2\Psi_{\Delta t}(\tilde{X}(s)).(AS_{\Delta t}^{k-\ell+1}e_j,\mathcal{D}_r^{e_j}\tilde{X}(s))\rangle \bigr]drdsdt.
\end{align*}
Let $\eta\in(\frac14,1)$, which allows us to use inequality~\eqref{eq:inequality1}. Thanks to Theorems~\ref{th:Du} and~\ref{th:D2u}, Lemmas~\ref{lem:borneXtilde-implicit} and~\ref{lem:malliavin},
\begin{align*}
|b_k^{2,2,3}|&\le \Delta t^2\sum_{\ell=0}^{k-1}\sum_{j\in\N}\bigl(\Delta tC_0(T,\Delta t_0,|x|_E)+\int_{k\Delta t}^{(k+1)\Delta t}\frac{C_\eta(T,\Delta t_0,|x|_E)}{(T-t)^{\eta}}dt\bigr)|AS_{\Delta t}^{k-\ell+1}e_j|_H\\
&\le C(T,\Delta t_0,|x|_E)\int_{k\Delta t}^{(k+1)\Delta t}\frac{1}{(T-t)^\eta}dt\ \Delta t\sum_{\ell=0}^{k-1}\|(-A)^{2\alpha}S_{\Delta t}^{k-\ell}\|_{\mathcal{L}(H)}\ \sum_{j\in\N}\frac{\lambda_j^{1-2\alpha}\Delta t}{1+\lambda_j\Delta t} \\
&\le C_\alpha(T,\Delta t_0,|x|_E)\int_{k\Delta t}^{(k+1)\Delta t}\frac{1}{(T-t)^\eta}dt \Delta t^{2\alpha}.
\end{align*}

Finally,
\[
|b_k^{2,2}|\le C_{\alpha}(T,\Delta t_0,|x|_E)\int_{k\Delta t}^{(k+1)\Delta t}\bigl(1+\frac{1}{t^{2\alpha}}+\frac{1}{(T-t)^{2\alpha}}\bigr)dt\ \Delta t^{2\alpha}.
\]

\subsubsection*{Treatment of $b_k^{2,3}$}

Applying Theorem~\ref{th:Du} and Lemma~\ref{lem:borneXtilde-implicit} directly gives
\[
|b_k^{2,3}|\le C_0(T,\Delta t_0,|x|_E)\Delta t^2.
\]

\subsubsection*{Treatment of $b_k^{2,4}$}

Using the Malliavin calculus duality formula~\eqref{eq:malliavin-duality} and the chain rule,
\begin{align*}
b_k^{2,4}&=\int_{k\Delta t}^{(k+1)\Delta t}\E\bigl[\sum_{j\in\N}\langle Du^{(\Delta t)}(T-t,\tilde{X}(t)),e_j\rangle \int_{k\Delta t}^{t}\langle D\Psi_{\Delta t}^{j}(\tilde{X}(s)),S_{\Delta t}dW(s)\rangle \bigr] dt\\
&=\int_{k\Delta t}^{(k+1)\Delta t}\int_{k\Delta t}^{t}\sum_{i,j\in\N}\E\bigl[D^2u^{(\Delta t)}(T-t,\tilde{X}(t)).(e_j,\mathcal{D}_s^{e_i}\tilde{X}(t))\langle D\Psi_{\Delta t}^{j}(\tilde{X}(s)),S_{\Delta t}e_i\rangle \bigr]dsdt\\
&=\int_{k\Delta t}^{(k+1)\Delta t}\int_{k\Delta t}^{t}\sum_{i\in\N}\E\bigl[D^2u^{(\Delta t)}(T-t,\tilde{X}(t)).(D\Psi_{\Delta t}(\tilde{X}(s))S_{\Delta t}e_i,\mathcal{D}_s^{e_i}\tilde{X}(t))\bigr]dsdt\\
&=\int_{k\Delta t}^{(k+1)\Delta t}\int_{k\Delta t}^{t}\sum_{i\in\N}\E\bigl[D^2u^{(\Delta t)}(T-t,\tilde{X}(t)).(D\Psi_{\Delta t}(\tilde{X}(s))S_{\Delta t}e_i,S_{\Delta t}e_i)\bigr]dsdt,
\end{align*}
indeed $\mathcal{D}_s\tilde{X}(t)=S_{\Delta t}$ for $k\Delta t\le s\le t\le (k+1)\Delta t$.

Let $\alpha\in(0,\frac12)$. Thanks to Theorem~\ref{th:D2u} and Lemma~\ref{lem:borneXtilde-implicit},
\begin{align*}
|b_k^{2,4}|&\le C_{0,\alpha}(T,|x|_E)\int_{k\Delta t}^{(k+1)\Delta t}\frac{1}{(T-t)^\alpha}dt\sum_{i\in\N}\frac{\Delta t}{\lambda_i^\alpha(1+\lambda_i\Delta t)^2}\\
&\le C_\alpha(T,\Delta t_0,|x|_E)\int_{k\Delta t}^{(k+1)\Delta t}\frac{1}{(T-t)^\alpha}dt \sum_{i=1}^{\infty}\frac{1}{\lambda_i^{1-\alpha}}\Delta t^{2\alpha}.
\end{align*}

\subsubsection*{Conclusion}
Gathering estimates above, for all $\alpha\in[0,\frac12)$,
\begin{align*}
\sum_{k=1}^{N-1}|b_k|\le C_\alpha(T,\Delta t_0,|x|_E)\Delta t^\alpha \sum_{k=1}^{N-1}\int_{k\Delta t}^{(k+1)\Delta t}\frac{1}{t^{\beta_1(\alpha)}(T-t)^{\beta_2(\alpha)}}dt,
\end{align*}
with two parameters $\beta_1(\alpha),\beta_2(\alpha)\in[0,1)$. Therefore
\[
\sum_{k=1}^{N-1}\int_{k\Delta t}^{(k+1)\Delta t}\frac{1}{t^{\beta_1(\alpha)}(T-t)^{\beta_2(\alpha)}}dt\le \int_{0}^{T}\frac{1}{t^{\beta_1(\alpha)}(T-t)^{\beta_2(\alpha)}}dt<\infty.
\]

This concludes the second part of the proof of Lemma~\ref{lem:akbkck}.

\subsubsection{Proof of Lemma~\ref{lem:akbkck}, Part 3}\label{sec:lem_ck}

It remains to treat $\sum_{k=1}^{N-1}|c_k|$. Let $\alpha\in(0,\frac12)$, and let $\epsilon\in(\frac12-\alpha,1-2\alpha)$ be an auxiliary parameter. Thanks to Theorem~\ref{th:D2u}, and to Lemma~\ref{lem:borneXtilde-implicit}, using $(-A)^{-\beta}e_j=\lambda_j^{-\beta}e_j$,
\begin{align*}
|c_k|&\le C_{2\alpha+\epsilon}(T,\Delta t_0,|x|_E)\int_{k\Delta t}^{(k+1)\Delta t}\frac{1}{(T-t)^{2\alpha+\epsilon}}dt\sum_{j\in\N}\frac{\lambda_j^{-2\alpha-\epsilon}\bigl(2\lambda_j\Delta t+(\lambda_j\Delta t)^2\bigr)}{(1+\lambda_j\Delta t)^2}\\
&\le C_{2\alpha+\epsilon}(T,\Delta t_0,|x|_E)\int_{k\Delta t}^{(k+1)\Delta t}\frac{1}{(T-t)^{2\alpha+\epsilon}}dt\sum_{j\in\N}\frac{\Delta t^{\alpha}}{\lambda_j^{\alpha+\epsilon}}\frac{2(\lambda_j\Delta t)^{1-\alpha}+(\lambda_j\Delta t)^{2-\alpha}}{(1+\lambda_j\Delta t)^2}.
\end{align*}
Note that $\sum_{j\in\N}\frac{1}{\lambda_j^{\alpha+\epsilon}}<\infty$, since $\alpha+\epsilon>\frac12$. In addition, $\underset{z\ge 0}\sup~\frac{z^\beta}{(1+z)^2}<\infty$ for $\beta\in[0,2]$. Thus
\[
\sum_{k=1}^{N-1}|c_k|\le C_{\alpha}(T,\Delta t_0,|x|_E)\Delta t^\alpha.
\]

\subsubsection{Conclusion}

Gathering the estimates on $\sum_{k=1}^{N-1}|a_k|$, $\sum_{k=1}^{N-1}|b_k|$ and $\sum_{k=1}^{N-1}|c_k|$, concludes the proof of Lemma~\ref{lem:akbkck}.

\subsection{Proof of Lemmas~\ref{lem:borneXtilde-implicit} and~\ref{lem:malliavin}}\label{sec:impl3}

\begin{proof}[Proof of Lemma~\ref{lem:borneXtilde-implicit}]

Let $\bigl(\omega_n\bigr)_{n\in\N}$ be given by
\[
\omega_n=\sum_{k=0}^{n-1}S_{\Delta t}^{n-k-1}\bigl(W((k+1)\Delta t)-W(k\Delta t)\bigr),
\]
which solves $\omega_{n+1}=S_{\Delta t}\omega_n+S_{\Delta t}\bigl(W((n+1)\Delta t)-W(n\Delta t)\bigr)$.

Then (see~\cite{Brehier_Goudenege:18}), there exists $C(T,\Delta t_0,M)\in(0,\infty)$ such that
\[
\E\bigl[\underset{0\le n\le N}\sup~|X_n|_E^M+\underset{0\le n\le N}\sup~|\omega_n|_E^M\bigr]\le C(T,\Delta t_0,M)(1+|x|_E)^M.
\]

Recall that $AS_{\Delta t}=\frac{S_{\Delta t}-I}{\Delta t}$, and that $\|S_{\Delta t}\|_{\mathcal{L}(E)}\le 1$.

For $n\Delta t\le t\le (n+1)\Delta t\le T$, using~\eqref{eq:scheme_implicit_tilde},
\begin{align*}
\bigl(\E|\tilde{X}(t)|_E^M\bigr)^{\frac{1}{M}}&\le \bigl(\E|X_n|_E^M\bigr)^{\frac{1}{M}}+2\frac{t-n\Delta t}{\Delta t}\bigl(\E|X_n|_E^M\bigr)^{\frac{1}{M}}+(t-n\Delta t)\bigl(\E|\Psi_{\Delta t}(X_n)|_E\bigr)^{\frac{1}{M}}\\
&+\bigl(\E|S_{\Delta t}(W(t)-W(n\Delta t)|_E^M\bigr)^{\frac{1}{M}}\\
&\le C(1+|x|_E+|x|_E^3)+\frac{(t-n\Delta t)^{\frac12}}{\Delta t^{\frac12}}\bigl(\E|S_{\Delta t}(W((n+1)\Delta t)-W(n\Delta t)|_E^M\bigr)^{\frac{1}{M}},
\end{align*}
and it remains to observe that $S_{\Delta t}(W((n+1)\Delta t)-W(n\Delta t)=\omega_{n+1}-S_{\Delta t}\omega_n$, and to use the estimate above.

This concludes the proof of Lemma~\ref{lem:borneXtilde-implicit}.
\end{proof}

\begin{proof}[Proof of Lemma~\ref{lem:malliavin}]
Introduce the following notation: for every $s\in[0,T]$, $\ell_s=\lfloor \frac{s}{\Delta t}\rfloor\in\left\{0,\ldots,N\right\}$.

If $\ell_s\ge k$, then $s\ge k\Delta t$, and by definition of the Malliavin derivative, $\mathcal{D}_sX_k=0$.

Note that $X_{\ell_s+1}=S_{\Delta t}\Phi_{\Delta t}(X_{\ell_s})+S_{\Delta t}(W((\ell_s+1)\Delta t)-W(\ell_s\Delta t))$, thus $\mathcal{D}_sX_{\ell_s+1}=S_{\Delta t}$.

Finally, if $k\ge \ell_s+1$, using the chain rule, for all $h\in H$,
\[
\mathcal{D}_s^hX_{k+1}=\mathcal{D}_s^h\bigl(S_{\Delta t}\Phi_{\Delta t}(X_k)\bigr)=S_{\Delta t}\Phi_{\Delta t}'(X_k)\mathcal{D}_s^hX_k,
\]
and since $\Phi_{\Delta t}$ is globally Lipschitz continuous, thanks to Lemma~\ref{lem:rappel},
\[
|\mathcal{D}_s^h(X_{k+1})|_H \le e^{\Delta t}|\mathcal{D}_s^ hX_k|\le e^T |\mathcal{D}_s^hX_{\ell_s+1}|_H\le e^T|h|_H.
\]

Finally, for $0\le s<k\Delta t\le t\le (k+1)\Delta t\le T$, and $h\in H$, using the chain rule and~\eqref{eq:scheme_implicit_tilde},
\begin{align*}
|\mathcal{D}_s^h\tilde{X}(t)|_H&\le |\mathcal{D}_s^h\tilde{X}(t)|_H+(t-n\Delta t)|AS_{\Delta t}\mathcal{D}_s^h\tilde{X}(t)|_H\\
&+(t-n\Delta t)|S_{\Delta t}\Psi_{\Delta t}'(X_k)\mathcal{D}_s^hX_k|_H\\
&\le (3+\Delta t|\Psi_{\Delta t}'(X_k)|_E)|\mathcal{D}_s^hX_k|_H,
\end{align*}
and using the estimate above concludes the proof of Lemma~\ref{lem:malliavin}.
\end{proof}

\section*{Acknowledgments}
The authors want to thank Arnaud Debussche for discussions, and for suggesting the approach to prove Lemma~\ref{lem:aux-eta}. They also wish to thank Jialin Hong and Jianbao Cui for helpful comments and suggestions to improve the presentation of the manuscript.

\end{document}